\documentclass{IEEEtran}
\usepackage{cite}
\usepackage{amsmath,amssymb,amsfonts}
\usepackage{algorithm}
\usepackage{algorithmic}
\usepackage{graphicx}
\usepackage{textcomp}
\usepackage{bm}
\usepackage{mathrsfs}
\usepackage{graphicx,cite}
\usepackage{amsmath}
\usepackage{amssymb}
\usepackage[dvips]{epsfig}
\usepackage{latexsym}
\usepackage{psfrag}
\usepackage{graphicx,subfigure}
\usepackage{hyperref}
\usepackage{color}
\usepackage{amsfonts,mathrsfs,bbm,amsmath,stmaryrd,amssymb,pifont}%
\usepackage{amscd,graphicx,array,dsfont,texdraw,tikz}%
\usetikzlibrary{arrows,decorations.markings,shapes,shadows,fadings}
\usepackage{bbm,mathbbol,mathrsfs}%
\usepackage{array}
\usepackage{enumerate}
\usepackage{times}
\usepackage{ntheorem}

\theoremseparator{.}
\newtheorem{definition}{Definition}
\newtheorem{lemma}{Lemma}
\newtheorem{theorem}{Theorem}
\newtheorem{corollary}{Corollary}
\newtheorem{remark}{Remark}
\newtheorem{assumption}{Assumption}
\def\diag{\text{diag}}

\allowdisplaybreaks

\begin{document}
%
%
%
%
%
%
%
%
%

\title{Distributed Nash Equilibrium Seeking in Consistency-Constrained Multi-Coalition Games}
\author{Jialing Zhou, Yuezu Lv, Guanghui Wen, \IEEEmembership{Senior Member, IEEE}, Jinhu L\"u, \IEEEmembership{Fellow, IEEE}, and Dezhi Zheng
	\thanks{Jialing Zhou is with the School of Automation, Nanjing University of Science and Technology, Nanjing 210094, China (e-mail: jialingz@njust.edu.cn).}
	\thanks{Yuezu Lv and Guanghui Wen are with the School of Mathematics, Southeast University, Nanjing 211189, China (e-mail: yzlv@seu.edu.cn, wenguanghui@gmail.com).}
	\thanks{Jinhu L\"u is with the School of Automation Science and Electrical Engineering, Beihang University, Beijing 100191, China (e-mail: jhlu@iss.ac.cn).}
	\thanks{Dezhi Zheng is with the Advanced Research Institute of Multidisciplinary Science, Beijing Institute of Technology, Beijing 100081, China (e-mail: zhengdezhi@bit.edu.cn).}
}

\maketitle

\begin{abstract}
Distributed Nash equilibrium (NE) seeking problem for multi-coalition games has attracted increasing attention 
in recent years, but the research mainly focuses on the case without agreement demand within coalitions.
This paper considers a class of networked games among multiple coalitions where each coalition contains multiple agents that cooperate to minimize the sum of their costs, subject to the demand of reaching an agreement on their state values. {Furthermore, the underlying network topology among the agents does not need to be balanced.} 
To  achieve {the goal of NE seeking within such a context}, two estimates are constructed for each agent, namely, an estimate of partial derivatives of the cost function and an estimate of global state values, based on which, an iterative state updating law is elaborately designed.
{ Linear convergence of the proposed algorithm is demonstrated.
It is shown that 
the consistency-constrained multi-coalition games investigated in this paper put the well-studied 
networked games among individual players and distributed optimization in a unified framework, 
and the proposed algorithm
can easily degenerate into solutions to these problems.
}

\end{abstract}


\section{Introduction}
The past decade has witnessed a growing interest on distributed computation of Nash Equilibria in networked games, where selfish interacting agents 
subject to 
network topologies
strive to make optimal decisions based on neighboring information 
in a distributed manner
\cite{Stankovic2012,Pavel_Auto2016,ye_cyber,dingzhengtao_cyber,fangxiao_cyber,yipeng_cyber,zhangkj_tcasii,jgs_cyber}. In this paper, we focus on a class of networked games, termed \textit{consistency-constrained multi-coalition games}. Specifically, the cost of a coalition is the sum of the costs of all the individual agents within the coalition, and the agents in each coalition collaborate to \textit{determine a common coalition decision} that minimizes the coalition cost. Such a problem formulation characterizes both the cooperation of agents that belong to the same coalition and the competition among coalitions, which has {potential} applications in the fields of economics, engineering, politics, military science, and so forth \cite{ZENG2017,ZENG2019,politic}. 
{
For example, in multi-party politics, multiple delegates with heterogeneous policy preferences in the same party should jointly decide a policy to be announced to compete with other parties to win the voting \cite{politic}.
Another example is the pricing competition of highly interchangeable products among multiple firms, where the subsidiaries belonging to the same firm cooperatively determine a uniform price for the product so as to achieve the profit maximization of the entire firm. 
}
Our goal is to develop a distributed algorithm for the agents in a consistency-constrained multi-coalition game that converges to the NE.

Existing studies concerning networked games among coalitions rather than individual agents are reported in \cite{two-network,two-network-switching,multi-cluster-ymj2018,multi-cluster-ymj2019,Pang20200829,Pang202012}.  
In \cite{two-network}, a two-coalition zero-sum game is introduced, where two adversarial coalitions have opposite objective regarding the optimization of a common function of the coalition states. Considering the requirement of agents in the same coalition agreeing on a common state, distributed dynamics of agent states that converge to the NE is developed under fixed undirected and weight-balanced directed topologies in \cite{two-network} and time-varying topologies in \cite{two-network-switching}. 
Then, networked games among multiple coalitions are further investigated in \cite{multi-cluster-ymj2018,multi-cluster-ymj2019,Pang20200829,Pang202012}, where continuous-time and discrete-time algorithms for NE computation are presented. However, in these studies, the consistency constraint within coalitions is removed, i.e., agents belong to the same coalition are permitted to have different states or make different decisions. When taking the consistency constraint into account, distributed NE seeking in the multi-coalition games becomes more challenging, and studies along this line are rare \cite{ZENG2017,ZENG2019,mengmin2020}. In \cite{ZENG2017}, a distributed continuous-time generalized NE seeking algorithm is proposed for consistency-constrained multi-coalition games, where the coalitions are assumed to have the same number of agents with the same topology for simplicity. The result is further extended to the case with different coalition graphs \cite{ZENG2019}. 
{
Note that the algorithms in \cite{ZENG2017,ZENG2019} are in the continuous-time domain, while it is preferable to design discrete-time algorithms for distributed NE seeking in consistency-constrained multi-coalition games for practical implementation. 
Such a discrete-time algorithm is proposed in \cite{mengmin2020} under undirected topologies, where 
the cost of each agent  is assumed to be independent of the states of the other agents in the same coalition, and  each coalition should designate one representative agent with the topology among the representative agents being connected. 
In practice, it is desirable to develop discrete-time algorithms with these restrictions relaxed. 
}

Motivated by the above discussions, we investigate the {discrete-time}  distributed NE computation for consistency-constrained multi-coalition games {with general individual objective functions} in general directed networks. {The main difficulty lies in determining appropriate variables to be estimated. And the general directed graph makes the problem more challenging, since the left eigenvector associated with the eigenvalue $1$ of the induced matrix of the unbalanced graph does not match the form of the coalition cost, which is the summation of the individual costs. 
	In this paper, the}
property of  NE in consistency-constrained multi-coalition games is analyzed, based on which, a new discrete-time distributed  NE seeking algorithm is designed. Specifically, 
two estimates are constructed for observing the partial derivations of the cost functions and the states of all the agents, and a new state updating law is elaborately designed for the agent states to converge to the NE while meeting the consistency constraint within coalitions. 
To make the algorithm effective in general unbalanced directed networks,   in the estimation of the partial derivations of the cost functions, each agent sends to its intra-coalition out-neighbors the weighted information, where the weights that satisfies a certain condition are determined by each agent in advance. Such a mechanism is inspired by the consensus protocol in \cite{kaicai} and the push-pull distributed optimization algorithm in \cite{pushi}.

The main contribution of this paper lies in the following aspects. (i) A fundamental architecture is proposed to develop a novel distributed discrete-time algorithm for the NE computation in  consistency-constrained multi-coalition games {with general individual objective functions} under general directed network topologies.
(ii) A unified framework is built for the studies of two-network zero-sum games \cite{two-network,two-network-switching}, networked games among individual players \cite{Pavel_Auto2016} and distributed optimization \cite{pushi,zhaoyu_opti}, and the proposed algorithm provides a unified solution to  these problems, since they are special cases of the consistency-constrained multi-coalition games. {(iii)  
A new distributed optimization algorithm is induced from the proposed distributed NE seeking algorithm, which reduces the communication burden compared with the existing push-pull algorithm in \cite{pushi}.
}


The remainder of the paper is organized as follows. The following section formulates the problem and analyzes the property of NE. Then, the distributed discrete-time NE seeking algorithm is proposed in Section \ref{sec.algorithm}, and convergence analysis is provided in Section \ref{sec.convergence}. Section \ref{sec.discussion} discusses the connection between the presented results with some related studies. Numerical simulations are carried out to verify the effectiveness of the proposed algorithm in \ref{sec.simulation}, and Section \ref{sec.conclusion} finally concludes the paper.

\textbf{Notation}. Symbols $\otimes$ and $\|\cdot\|$ denote the Kronecker product and the Euclidian norm, respectively. The set of real numbers is denoted by $\mathbb{R}$, and the set of $n$-dimensional real column vectors is denoted by $\mathbb{R}^n$. Symbols $\mathbb{Z}$ and $\mathbb{Z}^+$ are the sets of non-negative integers and positive integers, respectively. $I_n$ is the $n$-dimensional identity matrix. $\bm{1}_n$ denotes the $n$-dimensional column vector with all the entries being $1$.

\section{Problem Formulation}\label{SectionTwo}

\begin{figure}
\centering
\includegraphics[width=3.25in]{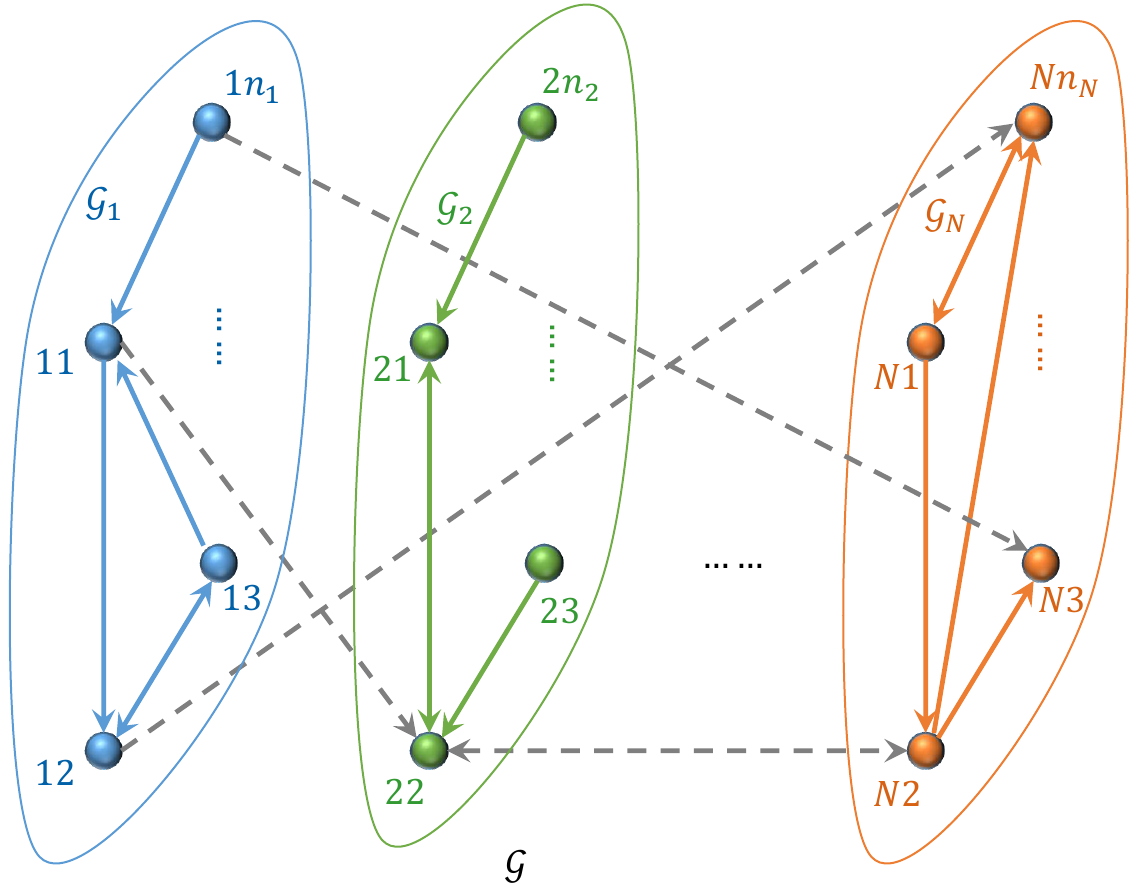}
\caption{The underlying communication topology of the multi-coalition game.}\label{fig.graph.sec.2}
\end{figure}

Consider a multi-coalition game of $N~(N\in\mathbb{Z}^+)$ coalitions indexed by $1,\cdots, N$ respectively. 
Coalition $i(i=1,\cdots,N)$ contains $n_i(n_i\in\mathbb{Z}^+)$ agents. Denote the agent set of coalition $i$ by $\mathcal{V}_i=\{i1,i2,\cdots,in_i\}$, where $ij$ represents the $j$th agent in coalition $i$.
Information spreads within each coalition and also across different coalitions. Denote the total number of the game participants by  $n_{\text{sum}}=\sum_{i=1}^Nn_i$. 
Figure \ref{fig.graph.sec.2} depicts the underlying communication topology among the $n_{\text{sum}}$ participants, which is represented by a directed graph $\mathcal{G}(\mathcal{V},\mathcal{E})$ with the node (agent) set $\mathcal{V}=\mathcal{V}_1\cup\cdots\cup\mathcal{V}_N$ and the edge (communication link) set $\mathcal{E}\subseteq \mathcal{V}\times \mathcal{V}$. A pair $(ij,pq)\in\mathcal{E}$ is an edge of $\mathcal{G}$ if agent $pq$ has access to the information of agent $ij$, and then agent $ij$ is called an \textit{in-neighbor} of agent $pq$, and agent $pq$ is called an \textit{out-neighbor} of agent $ij$.
A directed path from agent $i_1 j_1$ to agent $i_l j_l$ is 
a sequence of edges $(i_mj_m,i_{m+1}j_{m+1})\in\mathcal{E},m=1,\cdots,l-1$.
Graph $\mathcal{G}$ is strongly
connected if there exist directed paths from every node to every
other node.

Define  the induced subgraph  $\mathcal{G}_i(\mathcal{V}_i,\mathcal{E}_i)$ with the node set $\mathcal{V}_i$ and the edge set $\mathcal{E}_i\subseteq \mathcal{V}_i\times \mathcal{V}_i$. Obviously, $\mathcal{G}_i$ represents the communication topology among the agents within coalition $i$.

For each agent $ij\in\mathcal{V}$, define the in-neighbor set $\mathcal{N}_{ij}^{\text{in}}=\{pq|(pq,ij)\in\mathcal{E}\}$, and define the intra-coalition in-neighbor set $\mathcal{N}_{ij}^{i\text{-in}}=\{im|(im,ij)\in\mathcal{E}_i\}$  and intra-coalition out-neighbor set $\mathcal{N}_{ij}^{i\text{-out}}=\{im|(ij,im)\in\mathcal{E}_i\}$.

Define the adjacency matrix of graph $\mathcal{G}$ by $A=[a_{ij}^{pq}]_{n_{\text{sum}}\times n_{\text{sum}}}$, where $a_{ij}^{pq}$ denotes the element of $A$ on the $(\sum_{k=1}^{i-1}n_k+j)$th row and the $(\sum_{k=1}^{p-1}n_k+q)$th column, and
\begin{equation*}
a_{ij}^{pq}=\begin{cases}
1,&\text{if} ~(pq,ij)\in\mathcal{E},~pq\neq ij,\\
0,&\text{otherwise}.
\end{cases}
\end{equation*}
Define the Laplacian matrix of graph $\mathcal{G}$ by  $L=[l_{ij}^{pq}]_{n_{\text{sum}}\times n_{\text{sum}}}$,  where $l_{ij}^{pq}$ denotes the element of $L$ on the $(\sum_{k=1}^{i-1}n_k+j)$th row and the $(\sum_{k=1}^{p-1}n_k+q)$th column with
$l_{ij}^{ij}=\sum_{l=1}^{N}\sum_{m=1}^{n_l}a_{ij}^{lm}$ and $l_{ij}^{pq}=-a_{ij}^{pq},~ij\neq pq$.

Denote the adjacency matrix of graph $\mathcal{G}_i~(i=1,\cdots,N)$ by $A_i=[a_{ij}^{il}]_{{n_i}\times n_i}$ where $a_{ij}^{il}$ denotes the element of $A_i$ on the $j$th row and the $l$th column
 with
\begin{equation*}
a_{ij}^{il}=\begin{cases}
1,&\text{if} ~(il,ij)\in\mathcal{E}_i,~j\neq l,\\
0,&\text{otherwise}.
\end{cases}
\end{equation*}
Obviously,  $A_1,\cdots,A_N$ are the diagonal blocks of matrix $A$.

\begin{assumption}\label{assp.graph}
	The graph $\mathcal{G}$ is strongly connected, and all the sub-graphs $\mathcal{G}_i(i=1,\cdots,N)$ are strongly connected.
\end{assumption}

The state of agent $ij$ (the $j$th agent in coalition $i$) is denoted by $x_{ij}\in\mathbb{R}$. Let $\bm{x}_i=[x_{i1},x_{i2},\cdots,x_{in_i}]^{\rm{T}}\in\mathbb{R}^{n_i}$ and $\bm{x}=[\bm{x}_1^{\rm{T}},\bm{x}_2^{\rm{T}},\cdots,\bm{x}_{N}^{\rm{T}}]^{\rm{T}}\in\mathbb{R}^{n_{\textit{sum}}}$, which represent the collective state of coalition $i$ and the collective state of the multi-coalition game, respectively.
The cost of each agent may be influenced by all the game participants and the cost of each coalition is the sum of the costs of all the agents within the coalition. Let $f_{ij}({\bm{x}}):\mathbb{R}^{n_\text{sum}}\rightarrow\mathbb{R}$ be the cost function of agent $ij~(\forall~ij\in\mathcal{V})$, and then the cost function of each coalition can be formulated as
\begin{equation}\label{eq.cost function}
f_i(\bm{x})=\sum_{j=1}^{n_i}f_{ij}({\bm{x}}),\quad i=1,2,\cdots,N.
\end{equation}
In each coalition, the agents aim to achieve consensus of the state values, and at the meanwhile, minimize the cost of the coalition.
Define vector sets $$\Omega_i=\{\bm{x}_i{\in}\mathbb{R}^{n_i}|\bm{x}_i=\theta\bm{1}_{n_i},\theta\in\mathbb{R}\},~i=1,\cdots,N,$$ and $$\Omega{=}\{\bm{x}{\in}\mathbb{R}^{n_{\text{sum}}}|\bm{x}{=}[\theta_1\bm{1}_{n_1}^{\rm{T}},\theta_2\bm{1}_{n_2}^{\rm{T}},{\cdots},\theta_N\bm{1}_{n_N}^{\rm{T}}]^{\rm{T}},\theta_1,{\cdots},\theta_N{\in}\mathbb{R}\}.$$
Then, multi-coalition game  studied in this paper is formulated as follows.
\begin{equation}\label{eq.problem1}
\begin{aligned}
&\min_{{x}_{ij}}f_i(\bm{x})=\min_{{x}_{ij}}\sum_{j=1}^{n_i}f_{ij}({\bm{x}}),~\forall ij\in\mathcal{V},
\\
&\textit{s.t.}\quad \bm{x}_{i}\in\Omega_{i}.
\end{aligned}
\end{equation}
Before presenting the definition of NE, we denote $\bm{x}_{-i}=[\bm{x}_{1}^{\rm{T}},\cdots,\bm{x}_{i-1}^{\rm{T}},\bm{x}_{i+1}^{\rm{T}},\cdots,\bm{x}_{N}^{\rm{T}}]^{\rm{T}}$ and $(\bm{x}_{i},\bm{x}_{-i})\triangleq \bm{x}$
for notational simplicity.
\begin{definition}
An NE of the multi-coalition game (\ref{eq.problem1}) is a vector $\bm{x}^*=(\bm{x}_{i}^{*},\bm{x}_{-i}^*)\in\Omega$ with property that for each coalition $i=1,\cdots,N$,
\begin{equation*}
f_i(\bm{x}_{i}^{*},\bm{x}_{-i}^*){\leq}f_i(\bm{x}_{i},\bm{x}_{-i}^*),~ \forall \bm{x}_i\in\Omega_i.
\end{equation*}
\end{definition}
\begin{remark}
 Note that the consistency constraint of the agent states within each coalition can be rewritten as the equality constraint  $\mathcal{L}_i\bm{x}_i=0,~i=1,\cdots,N$. Using the method of Lagrange multipliers {\cite{book_nonlinear_optimization}}, one can derive that the NE of (\ref{eq.problem1}) $\bm{x}^*$ satisfies $$\frac{\partial f_i}{\partial \bm{x}_{i}}(\bm{x}^*)+\mathcal{L}_i^{\rm{T}}\bm{\lambda}_i^*=0,~\forall i=1,\cdots,N,$$ where $\bm{\lambda}_i^*,i=1,\cdots,N$ are the optimal Lagrange multipliers. This further implies that $$	\bm{1}_{n_i}^{\rm{T}}\frac{\partial f_i}{\partial \bm{x}_{i}}(\bm{x}^*)=0,~ \forall i=1,\cdots,N,$$
since  $\mathcal{L}_i\bm{1}_{n_i}=0$. 	$\hfill\blacksquare$
\end{remark}

{Next, we will show that the NE of the consistency-constrained N-coalition game has a special relationship with the NE of an N-player game.} 
Suppose that there is a virtual decision center in each coalition, which can obtain the full information of the coalition and determine a common state for the intra-coalition agents. Denote the common state by $y_i\in\mathbb{R}$ and let $\bm{y}=[y_1,y_2,\cdots,y_N]^{\rm{T}}\in\mathbb{R}^N$. Then, the cost function of coalition $i$ can be reformulated as a composite function  $g_i:\mathbb{R}^N\rightarrow\mathbb{R}$, induced by $f_i$ as follows:
\begin{equation}\label{eq.def.gi}
g_{i}(\bm{y})=f_{i}([y_1\bm{1}_{n_1}^{\rm{T}},y_2\bm{1}_{n_2}^{\rm{T}},\cdots,y_N\bm{1}_{n_N}^{\rm{T}}]^{\rm{T}}).
\end{equation}
View each virtual center as a virtual player.
And the multi-coalition game (\ref{eq.problem1}) turns to the game among the $N$ virtual players:
%
\begin{equation}\label{eq.problem2}
\min_{{y}_{i}}g_i(\bm{y}),\quad\forall i=1,2,{\cdots},N.
\end{equation}
Denote $\bm{y}_{-i}=[y_{1},\cdots,y_{i-1},y_{i+1},\cdots,y_{N}]^{\rm{T}}$ and $(y_i,\bm{y}_{-i})\triangleq\bm{y}$ for notational simplicity, and the definition of the NE in game (\ref{eq.problem2}) is given as follows.
\begin{definition}
	An NE of the game (\ref{eq.problem2}) is a vector $\bm{y}^*=({y}_{i}^{*},\bm{y}_{-i}^*)\in\mathbb
{R}^N$ with property that for each coalition $i=1,\cdots,N$,
	\begin{equation*}
g_i({y}_{i}^{*},\bm{y}_{-i}^*){\leq}g_i({y}_{i},\bm{y}_{-i}^*), ~\forall y_i\in\mathbb{R}.
\end{equation*}
\end{definition}
{One can directly derive from Definitions 1 and 2 that, $\bm{x}=[y_1\bm{1}_{n_1}^{\rm{T}},\cdots,y_N\bm{1}_{n_N}^{\text{     T}}]^{\rm{T}}$ is an NE of the N-coalition game (\ref{eq.problem1}), if and only if $\bm{y}=[y_1,\cdots,y_N]^{\rm{T}}$ is an NE of the N-player game (\ref{eq.problem2}).}
 
In this paper, we assume that the cost functions of the agents satisfy the following assumptions.

\begin{assumption}\label{assp.fij.lipschitz}
$\forall ij\in\mathcal{V}$, 
$f_{ij}(\cdot)$ is a convex $C^2$ function. Moreover, $\nabla f_{ij}(\cdot)$ is Lipschitz with the constant $l_{ij}>0$, i.e., $\|\nabla f_{ij}(\bm{a}){-}\nabla f_{ij}(\bm{b})\|{\leq}l_{ij}\|\bm{a}-\bm{b}\|,\forall\bm{a},\bm{b}\in\mathbb{R}^{n_\text{sum}}$.
\end{assumption}
Define the pseudo gradient  $$\mathcal{Q}(\bm{y})=[\frac{\partial g_1(\bm{y})}{\partial {y}_1},\frac{\partial g_2(\bm{y})}{\partial {y}_2},\cdots,\frac{\partial g_N(\bm{y})}{\partial {y}_N}]^{\rm{T}}\in\mathbb{R}^{N}.$$

\begin{assumption}\label{assp.Q}{
The pseudo gradient $\mathcal{Q}(\bm{y})$ is strongly monotone with the constant $l>0$, i.e.,} $(\bm a{-} \bm b)^{\rm{T}} (\mathcal{Q}(\bm a){-}\mathcal{Q}(\bm b)){\geq} l\|\bm a{-}\bm b\|^2,\forall\bm{a},\bm{b}\in\mathbb{R}^{N}$.
\end{assumption}
{
The above assumptions related to the cost functions are quite common in the study of NE seeking \cite{Pavel_Auto2016}, which together ensure the  existence and the uniqueness of  NE in game (\ref{eq.problem2}) and thus in game (\ref{eq.problem1}). Under Assumptions 2 and 3}, one can directly conclude from the above discussions that $$\bm{x}^*=[y_1^*\bm{1}_{n_1}^{\rm{T}},y_2^*\bm{1}_{n_2}^{\rm{T}},\cdots,y_N^*\bm{1}_{n_N}^{\rm{T}}]^{\rm{T}}.$$ 	

\begin{remark}
Since
	\begin{equation}\label{eq.proper.y*.g}
	\frac{\partial g_i}{\partial {y}_{i}}(\bm{y}^*)=0,\quad \forall i=1,\cdots,N,
	\end{equation}
	and
	\begin{equation}
	\begin{aligned}
	\frac{\partial g_i}{\partial {y}_{i}}(\bm{y})=&\sum_{j=1}^{n_i}\frac{\partial f_i}{\partial x_{ij}}([y_1\bm{1}_{n_1}^{\rm{T}},y_2\bm{1}_{n_2}^{\rm{T}},\cdots,y_N\bm{1}_{n_N}^{\rm{T}}]^{\rm{T}})\\
	&=\bm{1}_{n_i}^{\rm{T}}\frac{\partial f_i}{\partial \bm{x}_{i}}([y_1\bm{1}_{n_1}^{\rm{T}},y_2\bm{1}_{n_2}^{\rm{T}},\cdots,y_N\bm{1}_{n_N}^{\rm{T}}]^{\rm{T}}),
	\end{aligned}
	\end{equation}
	one has
	\begin{equation}\label{eq.proper.y*.f}
	\begin{aligned}
	&\bm{1}_{n_i}^{\rm{T}}	\frac{\partial f_i}{\partial \bm{x}_{i}}([y_1^*\bm{1}_{n_1}^{\rm{T}},y_2^*\bm{1}_{n_2}^{\rm{T}},\cdots,y_N^*\bm{1}_{n_N}^{\rm{T}}]^{\rm{T}})\\
	=&\bm{1}_{n_i}^{\rm{T}}\frac{\partial f_i}{\partial \bm{x}_{i}}(\bm{x}^*)=0,~~\forall i=1,\cdots,N,
	\end{aligned}
	\end{equation}
	which is consistent with the analysis in Remark 1. 	$\hfill\blacksquare$
\end{remark}	

In the networked multi-coalition game (\ref{eq.problem1}) investigated in this paper, each agent $ij\in\mathcal{V}$ has the knowledge of its own cost function $f_{ij}$ and state value $x_{ij}$; besides, it can obtain the information of its in-neighbours  via  the directed communication network $\mathcal{G}$.
The objective is to develop a distributed NE seeking algorithm under  the directed communication network $\mathcal{G}$, such that the collective agent state $\bm{x}$ will
converge to the NE  $\bm{x}^*$, given any initial state conditions. 

\section{Distributed discrete-time NE seeking algorithm}\label{sec.algorithm}

In this section, we present a new distributed NE seeking algorithm for the multi-coalition game.

First, each agent $ij\in\mathcal{V}$ chooses two sets of positive parameters $\left\{r_{ij}^{im}| im\in\mathcal{J}_{ij}^{i\textit{-in}}\right\}$ and $\left\{c_{im}^{ij}| im\in\mathcal{J}_{ij}^{i\textit{-out}}\right\}$, where $\mathcal{J}_{ij}^{i\textit{-in}}=\mathcal{N}_{ij}^{i\textit{-in}}\cup\{ij\}$, $\mathcal{J}_{ij}^{i\textit{-out}}=\mathcal{N}_{ij}^{i\textit{-out}}\cup\{ij\}$, which satisfy the following conditions:
\begin{equation}\label{eq.weight}
\begin{aligned}
\sum_{im\in\mathcal{J}_{ij}^{i\text{-in}}}r_{ij}^{im}=1, 	~r_{ij}^{im}>0, \forall im\in\mathcal{J}_{ij}^{i\text{-in}},\\
\sum_{im\in\mathcal{J}_{ij}^{i\text{-out}}}c_{im}^{ij}=1,~c_{im}^{ij}>0,\forall im\in\mathcal{J}_{ij}^{i\text{-out}}.
\end{aligned}
\end{equation}
The two sets of parameters 
are used as the weights of information that $ij$ received from its intra-coalition in-neighbors and sent to its intra-coalition out-neighbors respectively, based on which, the distributed NE seeking law  for each agent $ij\in\mathcal{V}$ is designed as:
\begin{align}
		x_{ij}(k+1)=&\sum_{im\in\mathcal{N}_{ij}^{i\text{-in}}}r_{ij}^{im}x_{im}(k){-}\frac{{\alpha}}{n_i}\sum_{m=1}^{n_i}\psi^{im}_{ij}(k),\label{eq.law.x.ij}\\
		\psi^{il}_{ij}(k+1)=&\sum_{im\in\mathcal{N}_{ij}^{i\text{-in}}}c_{ij}^{im}\psi_{im}^{il}(k)+\frac{\partial f_{ij}}{\partial x_{il}}\big(\bm{\xi}_{ij}(k+1)\big)\nonumber\\
		&{-}\frac{\partial f_{ij}}{\partial x_{il}}\big(\bm{\xi}_{ij}(k)\big),~~\forall il\in\mathcal{V}_i,\label{eq.law.psi.ijil}
\end{align} 
where {$\alpha$ is a positive constant to be determined later,} the initial value of $\psi_{ij}^{il}$ is set as $\psi_{ij}^{il}(t_0)=\frac{\partial f_{ij}}{\partial x_{il}}(\bm{\xi}_{ij}(t_0))$, and $\bm{\xi}_{ij}=[\xi_{ij}^{11},\xi_{ij}^{12},{\cdots},\xi_{ij}^{1{n_1}},\xi_{ij}^{21},{\cdots},\xi_{ij}^{2{n_2}},\cdots,\xi_{ij}^{N{n_N}}]\in\mathbb{R}^{n_{\text{sum}}}$ denotes the estimation of $\bm{x}$ by agent $ij$, with $\xi_{ij}^{pq},\forall pq\in\mathcal{V}$ governed by
\begin{equation}\label{eq.law.xi.ijpq}
	\begin{aligned}
\xi_{ij}^{pq}(k+1)&=\xi_{ij}^ {pq}(k){-}\frac{1}{{\gamma_{ij}^{pq}}}\bigg(\sum\limits_{lm\in\mathcal{N}_{ij}^{\text{in}}}\Big(\xi_{ij}^{pq}(k)\\
		&{-}\xi_{lm}^{pq}(k)\Big)+a_{ij}^{pq}\Big(\xi_{ij}^{pq}(k){-}x_{pq}(k)\Big)\bigg),
	\end{aligned}
\end{equation}
where {$\gamma_{ij}^{pq},~\forall ij,pq{\in}\mathcal{V}$ are constant parameters satisfying $\gamma_{ij}^{pq}>d_{ij}+a_{ij}^{pq}$  with  $d_{ij}$ denoting the in-degree of the agent $ij$, i.e., $d_{ij}=\sum_{lm\in\mathcal{N}_{ij}^{\text{in}}} a_{ij}^{lm}$.}

\begin{remark}

The analysis on the property of NE in the previous section has shown that $\bm{1}_{n_i}^{\rm{T}}\frac{\partial f_i}{\partial \bm{x}_{i}}(\bm{x}^*)=0,~~\forall i=1,\cdots,N$. To make $\bm{x}$ converge to $\bm{x}^*$,  in the state iterative design for each agent $ij\in\mathcal{V}$, we intuitively want to use the feedback of $\bm{1}_{n_i}^{\rm{T}}\frac{\partial f_i}{\partial \bm{x}_{i}}(\bm{x})$, which however, is global information that cannot be directly accessed by agent $ij$ in the networked game.
To bypass this obstacle, the leader-following consensus protocol (\ref{eq.law.xi.ijpq}) is used for each agent estimating the global state $\bm{x}$, and (\ref{eq.law.psi.ijil}) is designed to make $\psi_{ij}^{im}$ converge to $\tau_{ij}\frac{\partial f_i}{\partial {x}_{im}}(\bm{x})$ with $\tau_{ij}$ being a certain positive constant, such that $\sum_{m=1}^{n_i}\psi_{ij}^{im}$ will converge to $\tau_{ij}\bm{1}_{n_i}^{\rm{T}}\frac{\partial f_i}{\partial \bm{x}_{i}}(\bm{x})$. To conclude,
the proposed distributed NE seeking algorithm contains a state iterative law (\ref{eq.law.x.ij}) and two auxiliary laws (\ref{eq.law.psi.ijil}) and (\ref{eq.law.xi.ijpq}), where (\ref{eq.law.xi.ijpq}) is commonly used in distributed NE seeking  algorithms for global state estimation, while (\ref{eq.law.x.ij}) and (\ref{eq.law.psi.ijil}) are elaborately designed to solve the NE computation problem for the consistency-constrained multi-coalition game.
\end{remark}

\begin{remark}
The values of $x_{ij},\bm{\xi}_{ij}$ and $\bm{\psi}_{ij}\triangleq [\psi_{ij}^{i1},\cdots,\psi_{ij}^{in_i}]^{\rm{T}}$ are local to agent $ij$ at each iteration. Moreover, agent $ij$ obtains the information of  
$c_{ij}^{im}\bm{\psi}_{im},\forall im\in\mathcal{N}_{ij}^{i\text{-in}}$
and $\bm{\xi}_{pq}, x_{pq},\forall pq\in\mathcal{N}_{ij}^{\text{in}}$ 
from its in-neighbors via communication networks
{(it should be clarified that $\mathcal{N}_{ij}^{i\text{-in}}\subseteq\mathcal{N}_{ij}^{\text{in}}$, and thus the information of $x_{im}, \forall im\in\mathcal{N}_{ij}^{i\text{-in}}$, required in (\ref{eq.law.x.ij}), is included in the information of $x_{pq}, \forall pq\in\mathcal{N}_{ij}^{\text{in}}$).
Note that the weights $r_{ij}^{im}, \forall im\in\mathcal{N}_{ij}^{i\text{-in}}$ are local parameters that determined by agent $ij$, while
 the weights $c_{ij}^{im},\forall ij\in\mathcal{N}_{im}^{i\text{-out}}$ are determined by agent $im$.} Agent $im$ sends the value of  $c_{ij}^{im}\bm{\psi}_{im},\forall ij\in\mathcal{N}_{im}^{i\text{-out}}$ rather than that of $\bm{\psi}_{im}$ to its out-neighbors, and  this mechanism plays an important role for making the algorithm effective in general directed networks, which is inspired by the consensus protocol in \cite{kaicai} and the push-pull distributed optimization algorithm in \cite{pushi}.
There are different ways to determine the parameters 	$r_{ij}^{im},\forall im\in\mathcal{N}_{ij}^{i\textit{-in}}$ and $c_{im}^{ij},\forall im\in\mathcal{N}_{ij}^{i\textit{-out}}$. A simple choice is $r_{ij}^{im}=r_{ij}^{ij}=\frac{1}{|\mathcal{N}_{ij}^{i\text{-in}}|+1},\forall im\in\mathcal{N}_{ij}^{i\textit{-in}}$ and $c_{im}^{ij}=c_{ij}^{ij}=\frac{1}{|\mathcal{N}_{ij}^{i\text{-out}}|+1}, \forall im\in\mathcal{N}_{ij}^{i\textit{-out}}$.
\end{remark}

\begin{algorithm}
	\caption{Distributed NE seeking algorithm}
	Each agent $ij\in\mathcal{V}$ chooses two sets of parameters 
	$r_{ij}^{im},\forall im\in\mathcal{N}_{ij}^{i\textit{-in}}$ and $c_{im}^{ij},\forall im\in\mathcal{N}_{ij}^{i\textit{-out}}$ which satisfy (\ref{eq.weight}).

	Each agent $ij\in\mathcal{V}$ initializes with arbitrary $x_{ij}(t_0)\in\mathbb{R}$, $\bm{\xi}_{ij}(t_0)\in\mathbb{R}^{n_{\text{sum}}}$ and $\psi_{ij}^{il}(t_0)=\frac{\partial f_{ij}}{\partial x_{il}}(\bm{\xi}_{ij}(t_0))\in\mathbb{R},~\forall il\in\mathcal{V}_i$.
	
	\textbf{for} $k=0,1,\cdots,$ \textbf{do}
	\begin{enumerate}
		\item[] for each agent $ij\in\mathcal{V}$,
		\item[] agent $ij$ obtains $c_{ij}^{im}\psi_{im}^{i1},\cdots,c_{ij}^{im}\psi_{im}^{in_i},\forall im\in\mathcal{N}_{ij}^{i\text{-in}}$ and $\bm{\xi}_{pq}, x_{pq},\forall pq\in\mathcal{N}_{ij}^{\text{in}}$ from its in-neigbors.
		\item[]  agent $ij$ updates $\bm{\xi}_{ij}$ according to (\ref{eq.law.xi.ijpq}) ;
		\item[ ] agent $ij$ updates $x_{ij},\psi_{ij}^{i1},\cdots,\psi_{ij}^{in_i}$
		according to (\ref{eq.law.x.ij})(\ref{eq.law.psi.ijil}).
	\end{enumerate}
	\textbf{end for}
\end{algorithm}

To write the proposed NE seeking law in collective form, we 
define 	the following column vectors:
\begin{equation*}
\begin{aligned}
&\bm{\psi}_i=[\psi^{i1}_{i1},\psi^{i2}_{i1},{\cdots},\psi^{in_i}_{i1},\psi_{i2}^{i1},\cdots,\psi_{i2}^{in_i},\cdots,\psi_{in_i}^{in_i}]^{\rm{T}}\in\mathbb{R}^{n_i^2},\\
&\bm{\psi}=[\bm{\psi}_1^{\rm{T}},\bm{\psi}_2^{\rm{T}},{\cdots},\bm{\psi}_N^{\rm{T}}]^{\rm{T}}\in\mathbb{R}^{n_1^2+\cdots+n_N^2},\\
&\bm{\xi}_i=[(\bm{\xi}_{i1})^{\rm{T}},(\bm{\xi}_{i2})^{\rm{T}},\cdots,(\bm{\xi}_{i{n_i}})^{\rm{T}}]^{\rm{T}}\in\mathbb{R}^{n_in_{\text{sum}}},\\
&\bm{\xi}=[(\bm{\xi}_{1})^{\rm{T}},(\bm{\xi}_{2})^{\rm{T}},\cdots,(\bm{\xi}_{N})^{\rm{T}}]^{\rm{T}}\in\mathbb{R}^{n_{\text{sum}}^2},
\end{aligned}
\end{equation*} 
and the function $\mathcal{P}_i:\mathbb{R}^{n_in_\text{sum}}\rightarrowtail\mathbb{R}^{n_i^2}$:
\begin{equation*}
\begin{aligned}
{\mathcal{P}}_i(\bm{\xi}_{i})=&[\frac{\partial f_{i1}}{\partial x_{i1}}(\bm{\xi}_{i1}),\frac{\partial f_{i1}}{\partial x_{i2}}(\bm{\xi}_{i1}),{\cdots},\frac{\partial f_{i1}}{\partial x_{i{n_i}}}(\bm{\xi}_{i1}),\frac{\partial f_{i2}}{\partial x_{i1}}(\bm{\xi}_{i2}),\\
&\frac{\partial f_{i2}}{\partial x_{i2}}(\bm{\xi}_{i2}),
{\cdots},\frac{\partial f_{i2}}{\partial x_{i{n_i}}}(\bm{\xi}_{i2}),{\cdots},\frac{\partial f_{i{n_i}}}{\partial x_{i{n_i}}}(\bm{\xi}_{i{n_i}})]^{\rm{T}}\in\mathbb{R}^{n_i^2}.
\end{aligned}
\end{equation*}
Let $r_{ij}^{im}=0,\forall im\notin \mathcal{N}_{ij}^{i\text{-in}}$ and 
$c_{im}^{ij}=0,\forall im\notin\mathcal{N}_{ij}^{i\text{-out}}$. Define $R_i=[r_{ij}^{im}]\in\mathbb{R}^{n_i\times n_i}$ with $r_{ij}^{im}$ being the element on the $j$th row and the $m$th column. Similarly, define $C_i=[c_{ij}^{im}]\in\mathbb{R}^{n_i\times n_i}$ with $c_{ij}^{im}$ being the element on the $j$th row and the $m$th column.
Then, (\ref{eq.law.x.ij})-(\ref{eq.law.xi.ijpq}) can be respectively rewritten in the following compact form.
\begin{align}
\bm{x}_i(k+1)=&R_i\bm{x}_i(k){-}\frac{\alpha}{n_i}(I_{n_i}{\otimes}\bm{1}_{n_i}^{\rm{T}}){\bm{\psi}}_i(k),\label{eq.law.x.i}\\
{\bm{\psi}}_i(k+1)=&(C_i{\otimes} I_{n_i}){\bm{\psi}}_i(k)+{\mathcal{P}}_i(\bm{\xi}_i(k+1))\nonumber\\
&{-}{\mathcal{P}}_i(\bm{\xi}_i({k})),\label{eq.law.psi.i}\\
\bm{\xi}(k+1)
=&\bm{\xi}(k){-}\Gamma(\mathcal{L}{\otimes} I_{n_\text{sum}}+A_d)\big(\bm{\xi}(k)\nonumber\\
&{-}\bm{1}_{n_\text{sum}}{\otimes}\bm{x}(k)\big)\label{eq.law.xi}
\end{align}
where $\Gamma,A_d\in\mathbb{R}^{n_{\text{sum}}^2\times n_{\text{sum}}^2}$ are defined as
\begin{equation*}
\begin{aligned}
\Gamma=&\diag\{{\frac{1}{\gamma_{11}^{11}},\cdots,\frac{1}{\gamma_{11}^{Nn_N}},\frac{1}{\gamma_{12}^{11}},\cdots,\frac{1}{\gamma_{12}^{Nn_N}},\cdots,\frac{1}{\gamma_{Nn_N}^{Nn_N}}}\},\\
A_d=&\diag\{a_{11}^{11},{\cdots},a_{11}^{Nn_N},a_{12}^{11},{\cdots},a_{12}^{Nn_N},{\cdots},a_{Nn_N}^{Nn_N}\}.
\end{aligned}
\end{equation*}
By definitions, it is obvious that $R_i$ and $C_i$ respectively satisfies $R_i\bm{1}_{n_i}=\bm{1}_{n_i}$ and $\bm{1}_{n_i}^{\rm{T}}C_i=\bm{1}_{n_i}^{\rm{T}}$. We further define $u_i^{\rm{T}}$ as the left eigenvector of $R_i$ associated with the eigenvalue $1$, i.e., $u_i^{\rm{T}}R_i=u_i^{\rm{T}}$, which satisfies $u_i^{\rm{T}}\bm{1}_{n_i}=n_i$, and define $v_i$ as the right eigenvector of $C_i$ associated with the eigenvalue $1$, i.e., $C_iv_i=v_i$, which satisfies $\bm{1}_{n_i}^{\rm{T}}v_i=n_i$.

\section{Convergence Analysis}\label{sec.convergence}
For notational convenience, define
\begin{equation*}
\begin{aligned}
\bar{\bm{\psi}}_i&=\frac{1}{n_i}(\mathbf{1}_{n_i}^{\rm{T}}{\otimes}I_{n_i}\big){\bm{\psi}}_i\in\mathbb{R}^{n_i},\\
\bar{\mathcal{P}}_i(\cdot)&=\frac{1}{n_i}(\mathbf{1}_{n_i}^{\rm{T}}{\otimes}I_{n_i}){\mathcal{P}_i}(\cdot)\in\mathbb{R}^{n_i}
\end{aligned}
\end{equation*}
Since
$
\bm{\psi}_i(t_0)=\mathcal{P}_i(\bm{\xi}_i(t_0))$ and  $\bm{1}_{n_i}^{\rm{T}}C_i=\bm{1}_{n_i}^{\rm{T}}$, one can derive from  (\ref{eq.law.psi.i}) that
\begin{equation}\label{eq.bar.psi.P.i}
\bar{\bm{\psi}}_i(k)=\bar{\mathcal{P}}_i(\bm{\xi}_i(k)),\quad\forall k\in\mathbb{Z}.
\end{equation}
Noting by definition that
\begin{equation}\label{eq.bar.P.i.gradient}
\bar{\mathcal{P}}_i(\mathbf{1}_{n_i}{\otimes}\bm{x})=\frac{1}{n_i}\cdot\frac{\partial f_i}{\partial \bm{x}_i}(\bm{x}),
\end{equation}
the relation (\ref{eq.bar.psi.P.i})  is quite critical for the agents to estimate the partial derivations of the cost functions.
\subsection{Analysis on steady states}
In the following, we will provide some interpretations of the algorithm design from the perspective of steady states.
Suppose that the algorithm variables $\bm{x}_i$, $\bm{\psi}_i$ and $\bm{\xi}$ respectively converge to some points $\bm{x}_i(\infty)$, $\bm{\psi}_i(\infty)$ and $\bm{\xi}(\infty)$.
Then, from (\ref{eq.law.x.i})-(\ref{eq.law.xi}), the steady states satisfies
\begin{align}
&\bm{x}_i({\infty})=R_i\bm{x}_i(\infty){-}\frac{\alpha}{n_i}(I_{n_i}{\otimes}\bm{1}_{n_i}^{\rm{T}}){\bm{\psi}}_i(\infty),\label{eq.law.x.i.steady}\\
&{\bm{\psi}}_i(\infty)=(C_i{\otimes} I_{n_i}){\bm{\psi}}_i(\infty),\label{eq.law.psi.i.steady}\\
&\Gamma(\mathcal{L}{\otimes} I_{n_\text{sum}}+A_d)\big(\bm{\xi}(\infty){-}\bm{1}_{n_\text{sum}}{\otimes}\bm{x}(\infty)\big)=0.\label{eq.law.xi.steady}
\end{align}
From (\ref{eq.law.psi.i.steady}), one has 
$\bm{\psi}_i(\infty)=v_i\otimes {\delta_i},
$
where ${\delta_i}$ is a constant vector to be determined later. Since
\begin{equation*}
(\mathbf{1}_{n_i}^{\rm{T}}{\otimes}I_{n_i})\bm{\psi}_i(\infty)=(\mathbf{1}_{n_i}^{\rm{T}}{\otimes}I_{n_i})(v_i\otimes {\delta_i})=n_i{\delta_i},
\end{equation*}
one has $
{\delta_i}=\bar{\bm{\psi}}_i(\infty)$, and therefore,
\begin{equation}\label{eq.psi.i.infty}
\bm{\psi}_i(\infty)=v_i\otimes\bar{\bm{\psi}}_i(\infty).
\end{equation}
{Note that under Assumption 1, $\Gamma$ is a diagonal positive-definite matrix, and $(\mathcal{L}{\otimes} I_{n_\text{sum}}+A_d)$ is a non-singular $M$-matrix (see Lemma 1 in \cite{zhangkj_tcasii}). Therefore, $\Gamma(\mathcal{L}{\otimes} I_{n_\text{sum}}+A_d)$ is non-singular. Then,}
from (\ref{eq.law.xi.steady}), one has
\begin{equation}\label{eq.xi.i.infty}
\bm{\xi}(\infty)=\bm{1}_{n_\text{sum}}{\otimes}\bm{x}(\infty).
\end{equation}
Combining (\ref{eq.bar.psi.P.i}), (\ref{eq.bar.P.i.gradient}), (\ref{eq.psi.i.infty}) and (\ref{eq.xi.i.infty}) yields
\begin{equation}
\bm{\psi}_i(\infty)=v_i\otimes\left(\frac{1}{n_i}\cdot\frac{\partial f_i}{\partial \bm{x}_i}(\bm{x}(\infty)\right).
\end{equation}
Noticing that $u_i^{\rm{T}}R_i=u_i^{\rm{T}}$, multiplying $u_i^{\rm{T}}$ on both sides on equation (\ref{eq.law.x.i.steady}) yields
$u_i^{\rm{T}}(I_{n_i}{\otimes}\bm{1}_{n_i}^{\rm{T}}){\bm{\psi}}_i(\infty)=0.
$
Combining this with (\ref{eq.psi.i.infty}), one can derive that
\begin{equation}\label{eq.bar.psi.inf.2}
\bm{1}_{n_i}^{\rm{T}}\bar{\bm{\psi}}_i(\infty)=0,
\end{equation}
since $u_i^{\rm{T}}v_i\neq 0$.
From (\ref{eq.psi.i.infty}) and (\ref{eq.bar.psi.inf.2}), one has $(I_{n_i}{\otimes}\bm{1}_{n_i}^{\rm{T}}){\bm{\psi}}_i(\infty)=v_i\otimes(\bm{1}_{n_i}^{\rm{T}}\bar{\bm{\psi}}_i(\infty))=0$, substituting which back into (\ref{eq.law.x.i.steady}) yields
\begin{equation}
\bm{x}_i(t_{\infty})=R_i\bm{x}_i(\infty),
\end{equation}
which implies that
$
\bm{x}_i(\infty)=\tau_i\bm{1}_{n_i} ,
$
where $\tau_i$ is a constant to be determined later. Since
$
u_{i}^{\rm{T}}\bm{x}_i(\infty)=\tau_iu_i^{\rm{T}}\mathbf{1}_{n_i}=\tau_i n_i,
$
one has 
\begin{equation}\label{eq.x.i.infty}
\bm{x}_i(\infty)=\mathbf{1}_{n_i}\frac{u_i^{\rm{T}}\bm{x}_i(\infty)}{n_i}.
\end{equation}
which implies that $x_{ij}(\infty),\forall j\in\mathcal{V}_i$ reach an agreement with the consensus value  $\frac{u_i^{\rm{T}}\bm{x}_i(\infty)}{n_i}$.

Combining (\ref{eq.bar.psi.inf.2}) with (\ref{eq.bar.psi.P.i}) (\ref{eq.bar.P.i.gradient}) (\ref{eq.xi.i.infty}), one has
\begin{equation}
\bm{1}_{n_i}^{\rm{T}}\frac{\partial f_i}{\partial \bm{x}_i}(\bm{x}(\infty))=0,
\end{equation}
which implies  that
$
\bm{x}(\infty)=\bm{x}^*.
$

\subsection{Error system construction}

Define 
$$
\bar{x}_i=\frac{u_i^{\rm{T}}\bm{x}_i}{n_i}\in\mathbb{R},\quad\bar{X}_i=\bm{1}_{n_i}\bar{x}_i\in\mathbb{R}^{n_i}$$
and  $\bm{\bar{x}}=[\bar{x}_1,\cdots,\bar{x}_N]^{\rm{T}}\in\mathbb{R}^N$, $\bm{\bar{X}}=[\bar{X}_1^{\rm{T}},\cdots,\bar{X}_N^{\rm{T}}]^{\rm{T}}\in\mathbb{R}^{n_{\text{sum}}}$.
Based on the previous analysis,  we define the errors
\begin{equation*}
\begin{aligned}
\bm{e_{\psi_i}}(t)&=\bm{\psi}_i(t){-}v_i\otimes\bar{\bm{\psi}}_i(t)\in\mathbb{R}^{n_i^2},\\
{e_{\bar{x}_i}}(t)&=\bar{x}_i(t){-}y_i^*\in\mathbb{R},\\
	\bm{e_{x_i}}(t)&=\bm{x}_i(t){-}\bar{X}_i(t)\in\mathbb{R}^{n_i},\\
\bm{e_{\xi_i}}(t)&=\bm{\xi}_i(t){-}\bm{1}_{n_i}{\otimes}\bar{\bm{X}}(t)\in\mathbb{R}^{n_in_\text{sum}},\\
\end{aligned}
\end{equation*}
and $\bm{e_\psi}=[\bm{e_{\psi_1}}^{\rm{T}},\bm{e_{\psi_2}}^{\rm{T}},{\cdots},\bm{e_{\psi_N}}^{\rm{T}}]^{\rm{T}}$, 
  $\bm{e_{\bar{x}}}=[{e_{\bar{x}_1}},{\cdots},{e_{\bar{x}_N}}]^{\rm{T}}$,  $\bm{e_x}=[\bm{e_{x_1}}^{\rm{T}},\bm{e_{x_2}}^{\rm{T}},{\cdots},\bm{e_{x_N}}^{\rm{T}}]^{\rm{T}}$,  $\bm{e_\xi}=[\bm{e_{\xi_1}}^{\rm{T}},\bm{e_{\xi_2}}^{\rm{T}},{\cdots},\bm{e_{\xi_N}}^{\rm{T}}]^{\rm{T}}$.

\paragraph{The iteration of $\bm{e_{\psi_i}}$}
From the iteration of $\bm{{\psi_i}}$ in (\ref{eq.law.psi.i}), one can derive that:
\begin{equation}\label{eq.e_psi.i.k+1.k}
\begin{aligned}
&\bm{e_{\psi_i}}(k+1)\\
=&(C_i{\otimes} I_{n_i}){\bm{\psi}}_i(k)+{\mathcal{P}}_i(\bm{\xi}_i(k+1)){-}{\mathcal{P}}_i(\bm{\xi}_i({k}))\\
&{-}v_i\otimes\Big(\frac{1}{n_i}(\bm{1}_{n_i}^{\rm{T}}{\otimes} I_{n_i})\Big((C_i{\otimes} I_{n_i}){\bm{\psi}}_i(k)\\
&+{\mathcal{P}}_i(\bm{\xi}_i(k+1)){-}{\mathcal{P}}_i(\bm{\xi}_i({k}))\Big)\Big)\\
=&(C_i{\otimes} I_{n_i}){\bm{\psi}}_i(k){-}v_i{\otimes}\bar{\bm{\psi}}_i(k)+{\mathcal{P}}_i(\bm{\xi}_i(k+1)){-}{\mathcal{P}}_i(\bm{\xi}_i({k}))\\
&{-}v_i\otimes\Big(\frac{1}{n_i}(\bm{1}_{n_i}^{\rm{T}}{\otimes} I_{n_i})\big({\mathcal{P}}_i(\bm{\xi}_i(k+1)){-}{\mathcal{P}}_i(\bm{\xi}_i({k}))\big)\Big)\\
=&(C_i{\otimes} I_{n_i})\bm{e_{\psi_i}}(k)
+{\mathcal{P}}_i(\bm{\xi}_i(k+1)){-}{\mathcal{P}}_i(\bm{\xi}_i({k}))\\
&{-}\left(\frac{v_i\bm{1}_{n_i}^{\rm{T}}}{n_i}{\otimes} I_{n_i}\right)\big({\mathcal{P}}_i(\bm{\xi}_i(k+1)){-}{\mathcal{P}}_i(\bm{\xi}_i({k}))\big)\\
=&\left(\bar{C}_i{\otimes} I_{n_i}\right)\bm{e_{\psi_i}}(k)+\left(\bar{I}_{v_i}{\otimes} I_{n_i}\right)\big({\mathcal{P}}_i(\bm{\xi}_i(k+1)){-}{\mathcal{P}}_i(\bm{\xi}_i({k}))\big),
\end{aligned}
\end{equation}
where 
$
\bar{C}_i=C_i{-}\frac{v_i\bm{1}_{n_i}^{\rm{T}}}{n_i},
$
and 
$
\bar{I}_{v_i}=I_{n_i}{-}\frac{v_i\bm{1}_{n_i}^{\rm{T}}}{n_i},
$
and the last equality is obtained by noticing that 
$$(\frac{v_i\bm{1}_{n_i}^{\rm{T}}}{n_i}{\otimes} I_{n_i})\bm{e_{\psi_i}}(k)=0.$$
Under Assumption \ref{assp.graph}, $1$ is a simple eigenvalue of $C_i$, and the other eigenvalues, denoted by $\lambda_2,\cdots,\lambda_{n_i}$, are within the unit circle, i.e., $|\lambda_i|<1,\forall i=2,\cdots,n_i$. It is not difficult to verify that   
$0,\lambda_2,\cdots,\lambda_{n_i}$ are  the eigenvalues of $\bar{C}_i$. Therefore, $\bar{C}_i$ is a Schur matrix, which means that there exists a symmetric
positive definite matrix $W_{c_i}$ such that
$\bar{C}_i^{\rm{T}} W_{c_i}\bar{C}_i {-}W_{c_i} = {-}  I_{n_i}$.

\paragraph{The iteration of ${e_{\bar{x}_i}}$}
From the iteration of $\bm{x_i}$ in (\ref{eq.law.x.i}), one has
\begin{equation}\label{eq.e_bar.x.i.k+1.k}
\begin{aligned}
&{e_{\bar{x}_i}}(k+1)\\
=&\frac{u_i^{\rm{T}}}{n_i}\left(R_i\bm{x}_i(k){-}\frac{\alpha}{n_i}(I_{n_i}{\otimes}\bm{1}_{n_i}^{\rm{T}}){\bm{\psi}}_i(k)\right){-}{y}_i^*\\
=&{e_{\bar{x}_i}}({k}){-}\frac{\alpha}{n_i^2}u_i^{\rm{T}}(I_{n_i}{\otimes}\bm{1}_{n_i}^{\rm{T}}){\bm{\psi}}_i(k).
\end{aligned}
\end{equation}

\paragraph{The iteration of $\bm{e_{x_i}}$}
From the iteration of $\bm{x_i}$ in (\ref{eq.law.x.i}), one can also obtain
\begin{equation}\label{eq.e_x.i.k+1.k}
\begin{aligned}
&\bm{e_{x_i}}(k+1)\\
=&(I_{n_i}{-}\frac{\bm{1}_{n_i}u_i^{\rm{T}}}{n_i})\left(R_i\bm{x}_i(k){-}\frac{{\alpha}}{n_i}(I_{n_i}{\otimes}\bm{1}_{n_i}^{\rm{T}}){\bm{\psi}}_i(k)\right)\\
=&\bar{R}_{i}\bm{e_{x_i}}(k){-}\frac{\alpha}{n_i}\bar{I}_{u_i}{(I_{n_i}{\otimes}\bm{1}_{n_i}^{\rm{T}})}{\bm{\psi}}_i(k),
\end{aligned}
\end{equation}
where $\bar{R}_i=R_{i}{-}\frac{\bm{1}_{n_i}u_i^{\rm{T}}}{n_i}$, $\bar{I}_{u_i}=I_{n_i}{-}\frac{\bm{1}_{n_i}u_i^{\rm{T}}}{n_i}$.
Similar to the analysis on $\bar{C}_i$,
one can derive that $\bar{R}_i$ is a Schur matrix, which means that there exists a symmetric
positive definite matrix $W_{R_i}$ such that
$\bar{R}_i^{\rm{T}} W_{R_i}\bar{R}_i {-}W_{R_i} = {-}  I_{n_i}$.


\paragraph{The iteration of $\bm{e_\xi}$}
From (\ref{eq.law.x.i}), one has
\begin{equation}
\begin{aligned}
&\bar{{X}}_i(k+1)=\frac{\bm{1}_{n_i}u_i^{\rm{T}}}{n_i}\bm{x}_i(k+1)\\
=&\frac{\bm{1}_{n_i}u_i^{\rm{T}}}{n_i}\big(R_i\bm{x}_i(k){-}\frac{\alpha}{n_i}(I_{n_i}{\otimes}\bm{1}_{n_i}^{\rm{T}}){\bm{\psi}}_i(k)\big)\\
=&\bar{X}_i(k){-}\alpha h_i(k),
\end{aligned}
\end{equation}
where $h_i=\frac{\bm{1}_{n_i}u_i^{\rm{T}}}{n_i^2}(I_{n_i}{\otimes}\bm{1}_{n_i}^{\rm{T}}){\bm{\psi}}_i$, which can be further collectively rewritten as  
\begin{equation}\label{eq.bar.X}
\bar{X}(k+1)=\bar{X}({k}){-}\alpha h(k),
\end{equation}
where $h=[h_1^{\rm{T}},\cdots,h_N^{\rm{T}}]^{\rm{T}}$.
Combining (\ref{eq.law.xi})  and (\ref{eq.bar.X}), one can derive the iteration of $\bm{e_\xi}$
\begin{equation}\label{eq.e_xi.k+1.k}
\begin{aligned}
&\bm{e_\xi}(k+1)
=\bm{\xi}(k+1){-}\bm{1}_{n_{\text{sum}}}{\otimes}   \bar{\bm{X}}(k+1)\\
=&\bm{\xi}(k){-}\Gamma(\mathcal{L}{\otimes} I_{n_\text{sum}}+A_d)\big(\bm{\xi}(k){-}\bm{1}_{n_\text{sum}}{\otimes}\bm{x}(k)\big)\\
&{-}\bm{1}_{n_{\text{sum}}}{\otimes}\left(\bar{X}({k}){-}\alpha h(k)\right) \\
=&\bm{e_\xi}(k){-}\Gamma(\mathcal{L}{\otimes I_{n_\text{sum}}}+A_d)(\bm{e_\xi}(k){-}\bm{1}_{n_{\text{sum}}}{\otimes}\bm{e_{{x}}}(k))\\
&+\bm{1}_{n_\text{sum}}{\otimes}\big(\alpha h(k)\big)\\
=&\mathcal{M}\bm{e_\xi}(k)+\Gamma(\mathcal{L}{\otimes I_{n_\text{sum}}}+A_d)\big(\bm{1}_{n_{\text{sum}}}{\otimes}\bm{e_x}(k)\big)\\
&+\bm{1}_{n_\text{sum}}{\otimes}\big(\alpha h(k)\big),
\end{aligned}
\end{equation}
where $\mathcal{M}=I_{n_{\text{sum}}^2}{-}\Gamma(\mathcal{L}{\otimes}I_{n_{\text{sum}}}+A_d)$.
{Clearly,	 $1$ is not an eigenvalue of $\mathcal{M}$, due to the invertibility of  $\Gamma(\mathcal{L}{\otimes}I_{n_{\text{sum}}}+A_d)$ under Assumption \ref{assp.graph}  \cite{zhangkj_tcasii}. It follows  from the Gershgorin’s Circle Theorem that  $\mathcal{M}$  is a Schur matrix.}
Then, there exist a symmetric
positive definite matrices $W_\mathcal{M}$ such that
$\mathcal{M}^{\rm{T}} W_\mathcal{M}\mathcal{M} {-}W_\mathcal{M} = {-}  I_{n_{\text{sum}}^2}$.

\subsection{Convergence of the error system}
Now, we are {in the position} to present the main {results}.
\begin{theorem}
Suppose that Assumptions \ref{assp.graph}, \ref{assp.fij.lipschitz} and \ref{assp.Q} hold. The agent states will converge to the NE {with a linear rate} under the proposed algorithm (\ref{eq.law.x.ij}) (\ref{eq.law.psi.ijil}) and (\ref{eq.law.xi.ijpq}), by choosing $\alpha$ satisfying
\begin{equation}\label{eq.alpha}
\begin{aligned}
{\alpha}{\leq}\min\{\frac{l}{2N\sigma},\frac{\gamma_1}{8\sigma},\frac{\gamma_2}{8\sigma},1\},
\end{aligned}
\end{equation}
where
\begin{equation*}
\begin{aligned}
\gamma_1=&4\beta_{\bar{x}\psi},\\
\gamma_2=&4(\beta_{\bar x\xi}+\beta_{\psi\xi}\gamma_1),\\
\gamma_3=&4(\beta_{\psi x}\gamma_1+\beta_{\xi x}\gamma_2),\\
\sigma=&\max\limits_{i}\{b_{1i}\}+\gamma_2\max\limits_{i}\{b_{2i}\}+\gamma_3\max\limits_{i}\{b_{3i}\},
\end{aligned}
\end{equation*}
with
\begin{equation*}
\begin{aligned}
\beta_{\psi\xi}=&2\max_{i,j}\{(2\|\bar{C}_i^{\rm{T}}W_{c_i}\bar{I}_i\|^2+\|\bar{I}_i^{\rm{T}} W_{c_i} \bar{I}_i\|)l_{ij}^2\}\\
&\times\|\Gamma(\mathcal{L}{\otimes} I_{n_\text{sum}}+A_d)\|^2,\\
\beta_{\psi x}=&n_{\text{sum}}\beta_{\psi\xi},\\
\beta_{\xi x}=&n_{\text{sum}}(4\|\mathcal{M}^{\rm{T}} W_\mathcal{M}\|^2+2\|W_\mathcal{M}\|)\|\Gamma(\mathcal{L}{\otimes} I_{n_\text{sum}}+A_d)\|^2,\\
\beta_{\bar x\psi}=&\frac{2}{l}\max_{i}\{\frac{n_i^3\|u_i\|^2}{ (u_i^{\rm{T}}v_i)^2}\} ,\\
 \beta_{\bar x\xi}=&\frac{2\max\limits_{i}\{{n_i\sum_{j=1}^{n_i}l_{ij}^2}\}}{l},\\
 b_{0i}=&{n_i+(\frac{1}{n_i}+\max\limits_{i}\{n_i\})\|v_i\|^2{\sum_{j=1}^{n_i}l_{ij}^2}},\\
 b_{1i}=&\frac{\|u_i\|^2}{n_i u_i^{\rm{T}}v_i}b_{0i},\\
 b_{3i}=&\frac{1}{n_i^2}(2\|\bar{R}_{i}^{\rm{T}}W_{R_i}\bar{I}_{u_i}\|^2+\|\bar{I}_{u_i}^{\rm{T}}W_{R_i}\bar{I}_{u_i}\|)b_{0i},\\ b_{2i}=&{n_{\text{sum}}}\big(4\|\mathcal{M}^{\rm{T}} W_\mathcal{M}\|^2+2\|W_\mathcal{M}\|\big)\big\|\frac{\bm{1}_{n_i}u_i^{\rm{T}}}{n_i^2}\|^2b_{0i}.\\
\end{aligned}
\end{equation*}

\end{theorem}
\begin{proof}
Consider the following Lyapunov function
\begin{equation}\label{eq.Lya}
\begin{aligned}
V(t)=&V_{\bar{x}}(t)+\gamma_1 V_\psi(t)+\gamma_2 V_{\xi}(t)+\gamma_3 V_{x}(t),
\end{aligned}
\end{equation}
where 
\begin{equation}\label{eq.Lya.part}
\begin{aligned}
V_{\bar{x}}(t)&=\bm{e_{\bar{x}}}^{\rm{T}}(t)W_{\bar{x}}\bm{e_{\bar{x}}}(t),\\
V_\psi(t)&=\bm{e_\psi}(t)^{\rm{T}}W_{c}\bm{e_{\psi}}(t),\\
V_{\xi}(t)&=\bm{e_\xi}^{\rm{T}}(t)W_{\mathcal{M}}\bm{e_{\xi}}(t),\\
V_{x}(t)&=\bm{e_{x}}^{\rm{T}}(t)W_{R}\bm{e_{x}}(t),
\end{aligned}
\end{equation}
with 
\begin{equation*}
\begin{aligned}
W_{\bar{x}}&=\diag\{\frac{n_1^3}{\alpha u_1^{\rm{T}}v_1},\cdots,\frac{n_N^3}{\alpha u_N^{\rm{T}}v_N}\},\\
W_R&={\diag}\{W_{R_1},\cdots,W_{R_N}\},\\
W_c&=\diag\{W_{c_1}{\otimes}I_{n_1},\cdots,W_{c_N}{\otimes}I_{n_N}\}.
\end{aligned}
\end{equation*}
From Lemmas \ref{lemma.psi}-\ref{lemma.x} in the appendix, one can obtain that 
\begin{equation}
\begin{aligned}
&V(k+1){-}V(k)\\
{\leq}&{-}l\|\bm{e_{\bar{x}}}({k})\|^2{-}\frac{\gamma_1}{4}\|\bm{e_\psi}(k)\|^2{-}\frac{\gamma_2}{4}\|\bm{e_{\xi}}({k})\|^2-\frac{\gamma_3}{4}\|\bm{e_{x}}({k})\|^2\\
&+(\alpha \max\limits_{i}\{b_{1i}\}+\gamma_2\alpha^2\max\limits_{i}\{b_{2i}\}+\gamma_3\alpha^2\max\limits_{i}\{b_{3i}\})\\
&\times(\|{\bm{e_\psi}}(k)\|^2+\|\bm{e_{\xi}}(k)\|^2+N\|\bm{e_{\bar{x}}}(k)\|^2).
\end{aligned}
\end{equation}
Note from (\ref{eq.alpha}) that ${\alpha}{\leq}1$.
Then,
\begin{equation}
\begin{aligned}
&V(k+1){-}V(k)\\
{\leq}&{-}l\|\bm{e_{\bar{x}}}({k})\|^2{-}\frac{\gamma_1}{4}\|\bm{e_\psi}(k)\|^2{-}\frac{\gamma_2}{4}\|\bm{e_{\xi}}({k})\|^2{-}\frac{\gamma_3}{4}\|\bm{e_{x}}({k})\|^2\\
&+\alpha\sigma(N\|\bm{e_{\bar{x}}}(k)\|^2+{\bm{e_\psi}}(k)\|^2+\|\bm{e_{\xi}}(k)\|^2).
\end{aligned}
\end{equation}
Since $\alpha$ satisfies (\ref{eq.alpha}), one can further derive that
\begin{equation}\label{linear rate}
\begin{aligned}
&V(k+1){-}V(k)\\
{\leq}&{-}\frac{l}{2}\|\bm{e_{\bar{x}}}({k})\|^2{-}\frac{\gamma_1}{8}\|\bm{e_\psi}(k)\|^2{-}\frac{\gamma_2}{8}\|\bm{e_{\xi}}({k})\|^2{-}\frac{\gamma_3}{4}\|\bm{e_{x}}({k})\|^2\\
{\leq}&{-}\varepsilon V(k),
\end{aligned}
\end{equation}
where
\begin{equation}
\varepsilon=\min\{\frac{l}{2\|W_{\bar{x}}\|},\frac{1}{8\|W_{c}\|},\frac{1}{8\|W_{\mathcal{M}}\|},\frac{1}{4\|W_{R}\|}\},
\end{equation}
which means that $V$ will asymptotically converge to zero {with a linear rate $O((1-\varepsilon)^k)$}, and that $\bm{x}$  will asymptotically converge to $\bm{x}^*$ {with a linear rate $O((1-\varepsilon)^k)$}, i.e.,  $\lim\limits_{k\rightarrow\infty}\bm{x}(k)=\bm{x}^*$.	$\hfill\blacksquare$

\end{proof}

{
\begin{remark}
Let $h_{ij}(x_{ij},\bm{x}_{-i})\triangleq f_{ij}(x_{ij}\bm{1}_{n_i},\bm{x}_{-i})$, and the consistency-constrained multi coalition game (\ref{eq.problem1}) can be transformed into the game investigated in \cite{mengmin2020}, where the cost of each agent $ij\in\mathcal{V}$ is independent of the states of the other agents in coalition $i$. Thus, the problem formulated in this paper is more general than that in \cite{mengmin2020}.
Moreover, this  paper considers the directed communication topology among the game participants. It is worth mentioning that in \cite{mengmin2020}, each coalition should designate one representative agent, and the topology among the representative agents should be connected. Such a restriction is also removed in this paper, since 
none of the graphs in $\{\mathcal{G}^l=\{\mathcal{V}^l,\mathcal{E}^l\}|\mathcal{V}^l=\{1j_1,2j_2,\cdots,Nj_N\}~ \text{with}~ ij_i\in\mathcal{V}_i, ~ \mathcal{E}^l\subseteq\mathcal{V}^l\times\mathcal{V}^l~ \text{and}~ \mathcal{E}^l\subseteq\mathcal{E}\}$ is required to be connected; 
 while the tradeoff is that the algorithm is of higher order.
\end{remark}
}

{
Partially inspired by \cite{mengmin2020}, the proposed algorithm (\ref{eq.law.x.ij}) and (\ref{eq.law.psi.ijil}) can be further redesigned as the following reduced-order algorithm by letting $\phi_{ij}=\sum_{m=1}^{n_i}\psi^{im}_{ij}$:
\begin{align}
x_{ij}(k+1)=&\sum_{im\in\mathcal{N}_{ij}^{i\text{-in}}}r_{ij}^{im}x_{im}(k){-}\frac{{\alpha}}{n_i}\phi_{ij}(k),\label{eq.law.x.ij1}\\
\phi_{ij}(k+1)=&\sum_{im\in\mathcal{N}_{ij}^{i\text{-in}}}c_{ij}^{im}\phi_{im}(k)+\sum_{l=1}^{n_i}\frac{\partial f_{ij}}{\partial x_{il}}\big(\bm{\xi}_{ij}(k+1)\big)\nonumber\\
&{-}\sum_{l=1}^{n_i}\frac{\partial f_{ij}}{\partial x_{il}}\big(\bm{\xi}_{ij}(k)\big).\label{eq.law.phi.ij}
\end{align}
The following result is straightforward. 
\begin{corollary}
Suppose that Assumptions \ref{assp.graph}, \ref{assp.fij.lipschitz} and \ref{assp.Q} hold. The agent states will converge to the NE {with a linear rate} under the proposed algorithm (\ref{eq.law.x.ij1}) (\ref{eq.law.phi.ij}) and (\ref{eq.law.xi.ijpq}), by choosing $\alpha$ satisfying (\ref{eq.alpha}).
\end{corollary}
}


\section{Discussions}\label{sec.discussion}

\subsection{Comparison to distributed NE seeking in multi-coalition games without consistency constraints}
The distinct feature of the proposed distributed NE seeking algorithm, compared with that in \cite{Pang202012} for  multi-coalition games without consistency constraints, lies in (\ref{eq.law.x.ij}).

For the case without consistency constraints, $\frac{\partial f_i}{\partial x_{ij}}(\bm{x}^*)=0,\forall ij\in\mathcal{V}$. Then, each agent $ij\in\mathcal{V}$ only needs to update its state $x_{ij}$ with the feedback of $\frac{\partial f_i}{\partial x_{ij}}(\bm{x})$. Using $\psi_{ij}^{ij}$ to estimate $\tau_{ij}\frac{\partial f_i}{\partial x_{ij}}(\bm{x})$,
the iteration of $x_{ij}$ can be designed as :
\begin{equation}\label{xijwithout}
x_{ij}({k+1})=x_{ij}(k)-\beta\psi_{ij}^{ij}(k),
\end{equation}
where $\beta$ is a positive constant and $\psi_{ij}^{ij}$ is the auxiliary variable designed in (\ref{eq.law.psi.ijil}). The structure of (\ref{xijwithout}) is consistent with the algorithm in \cite{Pang202012}.

In the consistency-constrained multi-coalition games investigated in this paper, algorithm design for distributed NE computation is quite different and more complicated. In (\ref{eq.law.x.ij}),
weighted average state value of agent $ij$ and its in-neighbors is introduced mainly for intra-coalition state consensus, and the feedback of $\sum_{m=1}^{n_i}\psi_{ij}^{im}$ rather than that of $\psi_{ij}^{ij}$ is constructed mainly for driving the agent states to NE. 

\subsection{A connection between networked games and distributed optimization}
The foregoing sections have presented a fundamental algorithm design framework for distributed NE computation in consistency-constrained multi-coalition games, building a bridge between the networked games and the distributed optimization problems.

Specifically, for the special case that each coalition consists of only one agent, i.e., $n_i=1,\forall i=1,\cdots,N$, rewrite the agent set as $\mathcal{V}=\{1,\cdots,N\}$ for notational brevity. The consistency-constrained multi-coalition game studied in this paper degenerates into the networked game among $N$ individual agents, and the proposed updating law (\ref{eq.law.x.ij}), (\ref{eq.law.psi.ijil}) and (\ref{eq.law.xi.ijpq}) turns into
\begin{equation}\label{xijonecoalition}
\begin{aligned}
x_{i}({k+1})=&x_{i}(k)-\alpha\psi_{i}^{i}(k),\\
\psi_{i}^i({k+1})=&\psi_i^i(k)+\frac{\partial f_i}{\partial x_i}(\bm{\xi}_i({k+1}))-\frac{\partial f_i}{\partial x_i}(\bm{\xi}_i({k})),\\
\xi_i^p({k+1})=&\xi_i^p(k)-\frac{1}{d_i+a_i^p}\bigg(\sum_{l\in\mathcal{N}_i^{\textit{in}}}(\xi_i^p(k)-\xi_l^p(k))\\
&~~~~~~~+a_i^p(\xi_i^p(k)-x_p(k))\bigg),~~\forall p\in\mathcal{V},
\end{aligned}
\end{equation}
where $\bm{\xi}_i=[\xi_i^1,\cdots,\xi_i^N]^{\rm{T}}$, and $\psi_i^i(0)=\frac{\partial f_i}{\partial x_i}(\bm{\xi}_i({0}))$. This is consistent with the discrete-time algorithm in \cite{ye_cyber} by noticing that
$$\psi_i^i(k)\equiv\frac{\partial f_i}{\partial x_i}(\xi_i({k})).$$

For another special case that there is only one coalition, i.e., $N=1$, rewrite the agent set as $\mathcal{V}=\mathcal{V}_1=\{1,\cdots,n\}$ for notational brevity. The consistency-constrained multi-coalition game studied in this paper degenerates into the distributed optimization problem
\begin{equation}\label{problem_opti}
\begin{aligned}
&\min_{x_{i}}\sum_{i=1}^{n}f_{i}(x_{1},\cdots,x_{n}),\\
&{\text{s.t.}}~x_{1}=x_{2}=\cdots=x_{n},
\end{aligned}
\end{equation} 
and the proposed algorithm turns into
\begin{equation}\label{algorithm3}
	\begin{aligned}
x_{i}(k+1)=&\sum_{m\in\mathcal{N}_{i}^{\text{in}}}r_{i}^{m}x_{m}(k){-}\frac{{\alpha}}{n}\sum_{m=1}^{n}\psi^{m}_{i}(k),\\
\psi^{l}_{i}(k+1)=&\sum_{m\in\mathcal{N}_{i}^{\text{in}}}c_{i}^{m}\psi_{m}^{l}(k)+\frac{\partial f_{i}}{\partial x_{l}}\big(\bm{\xi}_{i}(k+1)\big)\\
&{-}\frac{\partial f_{i}}{\partial x_{l}}\big(\bm{\xi}_{i}(k)\big),~~\forall l\in\mathcal{V},\\
\xi_{i}^{p}(k+1)=&\xi_{i}^ {p}(k){-}\frac{1}{d_{i}+a_{i}^{p}}\bigg(\sum\limits_{l\in\mathcal{N}_{i}^{\text{in}}}\Big(\xi_{i}^{p}(k){-}\xi_{l}^{p}(k)\Big)\\
&+a_{i}^{p}\Big(\xi_{i}^{p}(k){-}x_{p}(k)\Big)\bigg),~~\forall p\in\mathcal{V},
	\end{aligned}
\end{equation} 
where $\bm{\xi}_i=[\xi_i^1,\cdots,\xi_i^N]^{\rm{T}}$, and $\psi_i^l(0)=\frac{\partial f_i}{\partial x_l}(\bm{\xi}_i({0}))$.

{
Let $h_i(x_i)\triangleq f_i(x_i,\cdots,x_i)$.	
Then, problem (\ref{problem_opti}) has the same solution with the following well-studied distributed optimization problem
\begin{equation}\label{problem_opti2}
\begin{aligned}
&\min_{x_{i}}\sum_{i=1}^{n}h_{i}(x_{i}),\\
&{\text{s.t.}}~x_{1}=x_{2}=\cdots=x_{n}.
\end{aligned}
\end{equation}
In (\ref{problem_opti2}), the local cost function $h_i(x_i)$ of agent $i$ is no longer contingent on the states of other agents, indicating that the state estimation $\xi_i^p,\forall p\in\mathcal{V}$ is not needed. Then, the algorithm (\ref{algorithm3}) turns into the following form to solve the problem (\ref{problem_opti2}):
\begin{equation}\label{algorithm4}
\begin{aligned}
x_{i}(k+1)=&\sum_{m\in\mathcal{N}_{i}^{\text{in}}}r_{i}^{m}x_{m}(k){-}\frac{{\alpha}}{n}\sum_{m=1}^{n}\psi^{m}_{i}(k),\\
\psi^{l}_{i}(k+1)=&\sum_{m\in\mathcal{N}_{i}^{\text{in}}}c_{i}^{m}\psi_{m}^{l}(k)+\frac{\partial f_{i}}{\partial x_{l}}\big(x_{i}(k+1),\cdots,x_i(k+1)\big)\\
&{-}\frac{\partial f_{i}}{\partial x_{l}}\big(x_{i}(k),\cdots,x_i(k)\big),~~\forall l\in\mathcal{V}.\\
\end{aligned}
\end{equation} 
By defining $y_i=\sum_{m=1}^{n}\psi_i^m$ and noting
$\frac{dh_i(x_i)}{dx_i}=\sum_{l=1}^{n}\frac{\partial f_i(x_i,\cdots,x_i)}{\partial x_i}$, one can finally derive from (\ref{algorithm4}) the following algorithm for problem (\ref{problem_opti2}).
\begin{equation}\label{algorithm5}
\begin{aligned}
x_{i}(k+1)=&\sum_{m\in\mathcal{N}_{i}^{\text{in}}}r_{i}^{m}x_{m}(k){-}\frac{{\alpha}}{n}y_i(k),\\
y_{i}(k+1)=&\sum_{m\in\mathcal{N}_{i}^{\text{in}}}c_{i}^{m}y_m(k)+\frac{d h_{i}\big(x_{i}(k+1)\big)}{d x_{i}}{-}\frac{dh_{i}\big(x_{i}(k)\big)}{d x_{i}}.
\end{aligned}
\end{equation} 
\begin{remark}
The difference between the derived algorithm (\ref{algorithm5}) and the push-pull algorithm in \cite{pushi} lies in the second term of the updating of $x_i$. Specifically, in the push-pull algorithm \cite{pushi}, the updating of $x_i$ uses the values of $y_m,\forall m\in\mathcal{N}_i^{\text{in}}$. As a result, the push-pull algorithm \cite{pushi} requires that agent $j,j\in\mathcal{V}$ transmits the values of $x_j$, $y_j$, and $c_m^jy_j$ to its out-neighbor $m,\forall m\in\mathcal{N}_j^{\text{out}}$. In contrast, only the values of $x_j$ and $c_m^jy_j$ need to be transmitted from agent $j$ to its out-neighbor $m,\forall m\in\mathcal{N}_j^{\text{out}}$ in
the derived algorithm (\ref{algorithm5}). 
Thus, compared with the push-pull algorithm in \cite{pushi}, the derived algorithm (\ref{algorithm5}) reduces the communication burden since less information flows are required.
\end{remark} 
}

\section{Numerical simulations}\label{sec.simulation}
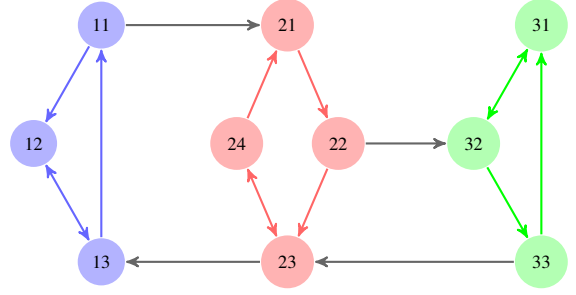
\begin{figure}
\centering
\begin{tikzpicture}[->,>=stealth',auto,node distance=2cm,scale=0.9,
thick, 
node1/.style={circle,fill=blue!30,align=center,minimum height=10pt,minimum  width=10pt,font=\fontsize{7}{7}\selectfont}, 
node2/.style={circle,fill=red!30,align=center,minimum height=20pt,minimum  width=20pt,font=\fontsize{7}{7}\selectfont},
node3/.style={circle,fill=green!30,align=center,minimum height=20pt,minimum  width=20pt,font=\fontsize{7}{7}\selectfont}]	

\node[node1] (11) at(1,0) {11};
\node[node1] (12) at(0,-1.75) {12};
\node[node1] (13) at(1,-3.5) {13};

\node[node2] (21) at(3.75,0) {21};
\node[node2] (24) at(3,-1.75) {24};
\node[node2] (22) at(4.5,-1.75) {22};
\node[node2] (23) at(3.75,-3.5) {23};

\node[node3] (31) at(7.5,0) {31};
\node[node3] (32) at(6.5,-1.75) {32};
\node[node3] (33) at(7.5,-3.5) {33};

\draw[->,blue!60] (11) -- (12);
\draw[<->,blue!60] (12) -- (13);
\draw[->,blue!60] (13) -- (11);

\draw[->,red!60] (21) -- (22);
\draw[->,red!60] (22) -- (23);
\draw[<->,red!60] (23) -- (24);
\draw[->,red!60] (24) -- (21);

\draw[->,green] (32) -- (33);
\draw[<->,green] (31) -- (32);
\draw[->,green] (33) -- (31);

\draw[->,black!60] (11) -- (21);
\draw[->,black!60] (23) -- (13);
\draw[->,black!60] (22) -- (32);
\draw[->,black!60] (33) -- (23);

\end{tikzpicture}
\caption{The unbalanced topology of the multi-coalition game.}\label{fig.graph.simu}
\end{figure}
Numerical simulations are provided in this section to verify the theoretical analysis.
Consider three coalitions  which contains three, four, three agents respectively, i.e., $N=3$, $n_1=3$, $n_2=4$, $n_3=3$. The communication topology is depicted in Fig. \ref{fig.graph.simu}. The cost function of each agent $ij\in\mathcal{V}$ is 
\begin{equation*}
f_{ij}=m_{ij}(x_{ij}^2-s_{ij}x_{ij})-h_{ij}x_{ij}(\mathbf{1}^{\rm{T}}\bm{x}) ,
\end{equation*}
where $m_{ij},s_{ij},h_{ij}$ are positive constants set as 
$m_{11}=3$, $m_{12}=11$,
$m_{13}=22$, $m_{21}=m_{22}=2$, $m_{23}=64$,
$m_{24}=8$, $m_{31}=60$,  
$m_{32}=m_{33}=4$,
$s_{11}=s_{12}=s_{13}=10$,
$s_{21}=s_{22}=s_{23}=s_{24}=50$,  $s_{31}=s_{32}=s_{33}=20$,
$h_{11}=0.35$, $h_{12}=0.25$, $h_{13}=0.15$, $h_{21}=0.2$, $h_{22}=0.1$, $h_{23}=0.05$, $h_{24}=0.25$,  $h_{31}=0.02$, $h_{32}=0.08$, $h_{33}=0.2$. By direct calculation, one can obtain the NE
$\bm{y}^*=[6.837,26.026,10.412]^{\rm{T}}$ and $\bm{x}^*=[6.837,6.837,6.837,26.026,26.026,26.026,26.026,10.412,$ $10.412,10.412]^{\rm{T}}$. 

The initial state is 
$\bm{x}(t_0)=[0,10,20,0,10,20,30,0,10,20]^{\rm{T}}$ and the algorithm parameter is chosen as
$\alpha=0.02$. The simulation result is shown in Fig. \ref{fig.x}, demonstrating the effectiveness of the proposed algorithm. {It can be seen from Fig. \ref{fig.x} that the convergence of the consensus error is very fast. However, this does not indicate that the consensus is reached before seeking NE. Note that in the proposed updating law of $x_{ij}$ in (\ref{eq.law.x.ij}), the first term $\sum_{im\in\mathcal{N}_{ij}^{i\text{-in}}}r_{ij}^{im}x_{im}(k)$ is mainly for driving the agent states in coalition $i$ to reach consensus, and the second term $-\frac{{\alpha}}{n_i}\sum_{m=1}^{n_i}\psi^{im}_{ij}(k)$ is mainly for making the agent states converge to the NE. Therefore, the consensus process and NE seeking process are going on at the same time. 
	 The main reason for the phenomenon of fast consensus is that the parameter $\alpha$ is relatively small.}

{Figure \ref{fig.x.compare} depicts the state trajectories of the agents under the algorithm (\ref{xijwithout}) for distributed NE seeking in multi-coalition games without consistency constraints, which indicates that the intra-coalition state-consensus cannot be achieved, and the final states of the agents are quite different from those in the consistency-constrained multi-coalition game.}

\begin{figure}
	\centering
	\includegraphics[width=3.4in]{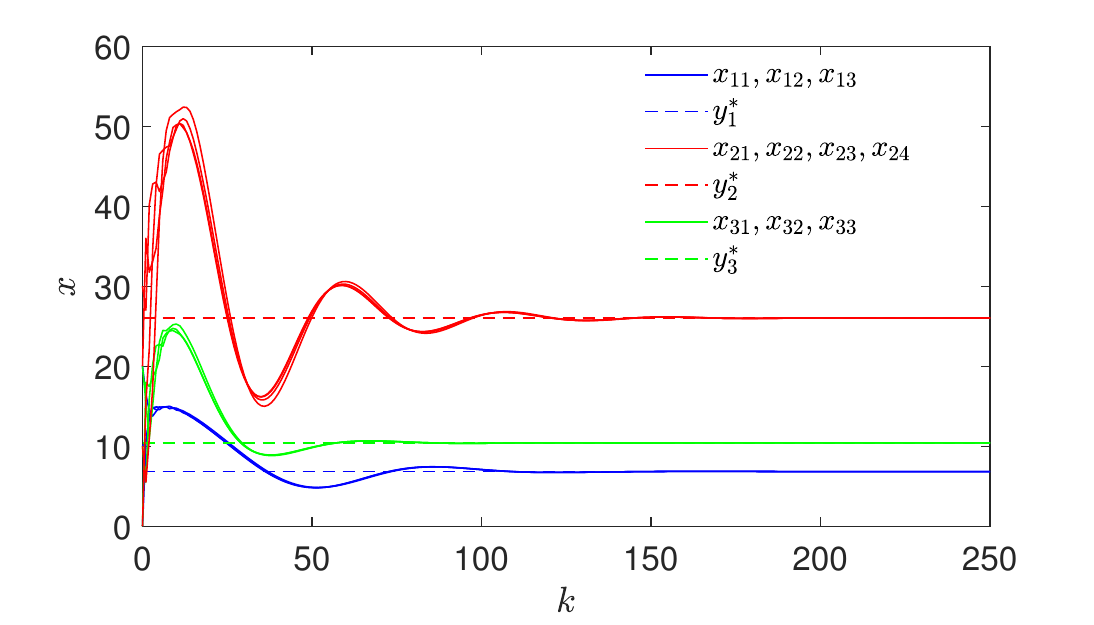}
	\caption{The state trajectories under the proposed algorithm (\ref{eq.law.x.ij}).}
	\label{fig.x}
\end{figure}

\begin{figure}
	\centering
	\includegraphics[width=3.4in]{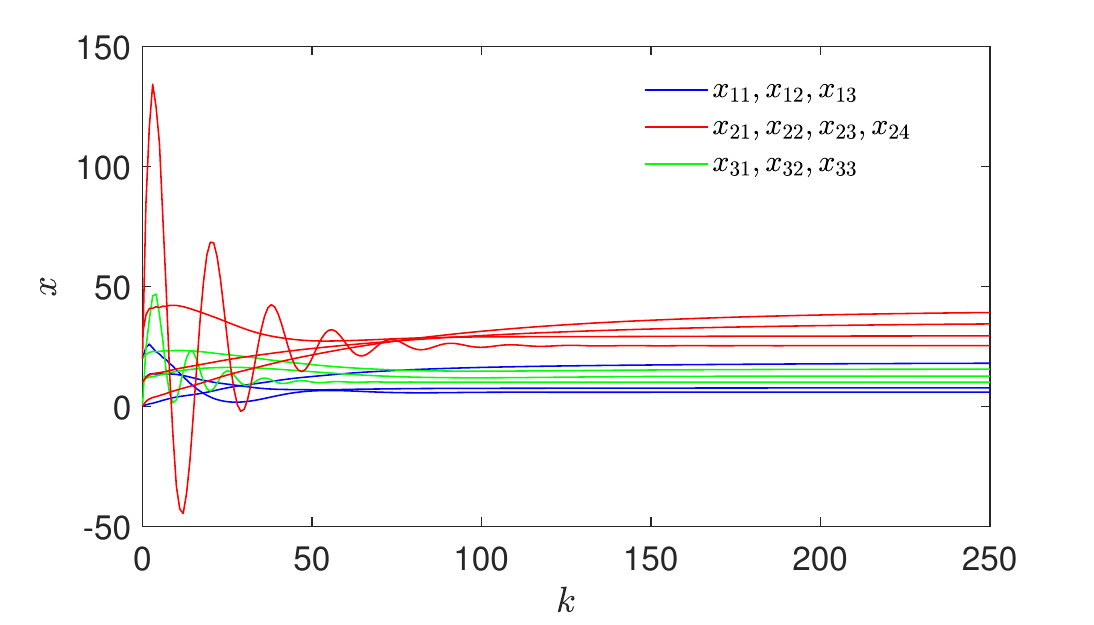}
	\caption{{The state trajectories under the algorithm (\ref{xijwithout}).}}
	\label{fig.x.compare}
\end{figure}

\vspace{1.5ex}

\section{Conclusion}\label{sec.conclusion}
The consistency-constrained multi-coalition games, where  each coalition contains multiple agents that aim to minimize the sum of their costs and at the meanwhile reach an agreement on their state values, have been investigated in this paper. %
A  new distributed discrete-time NE computation algorithm is developed under general directed network topologies.
The derived results provide a fundamental and unified framework for the studies of networked games and distributed optimization.
It is not difficult to extend the results to further develop distributed NE computation algorithms for consistency-constrained multi-coalition games with non-smooth cost functions {and extra state constraints. Future work will be done to further investigate reduced-order distributed NE seeking algorithm of consistency-constrained multi-coalition games with time-varying objective functions under switching topologies.} 
\appendix
\begin{lemma}\label{lemma.psi}
Suppose Assumptions \ref{assp.graph} and \ref{assp.fij.lipschitz} hold. Then, the function $V_\psi (t)$ defined in (\ref{eq.Lya.part}) satisfies
	\begin{equation}\label{eq.lemma.V_psi}
	\begin{aligned}
	&V_{\psi}(k+1){-}V_{\psi}(k)\\
	{\leq}
	&{-}\frac{1}{2}\|\bm{e_\psi}(k)\|^2+\beta_{\psi \xi}\|\bm{e_\xi}(k)\|^2+\beta_{\psi x}\|\bm{e_x}(k)\|^2.
	\end{aligned}
	\end{equation}
\end{lemma}
\begin{proof}
From the iteration of $\bm{e_{\psi_i}}$ derived in (\ref{eq.e_psi.i.k+1.k}), one can obtain that
\begin{equation*}\label{eq.V_psi.k+1.k.0}
\begin{aligned}
&V_{\psi}(k+1){-}V_{\psi}(k)\\
=&\sum_{i=1}^N\bigg(\bm{e_{\psi_i}}(k)^{\rm{T}}\big((\bar{C}_i^{\rm{T}}W_{c_i}\bar{C}_i-W_{c_i}){\otimes}I_{n_i}\big)\bm{e_{\psi_i}}(k)\\
&+2\bm{e_{\psi_i}}(k)^{\rm{T}}(\bar{C}_i^{\rm{T}}W_{c_i}\bar{I}_i{\otimes}I_{n_i})\left(\mathcal{P}_i(\bm{\xi}_i(k+1)){-}\mathcal{P}_i(\bm{\xi}_i({k}))\right)\\
&+\left(\mathcal{P}_i(\bm{\xi}_i(k+1)){-}\mathcal{P}_i(\bm{\xi}_i({k}))\right)^{\rm{T}}(\bar{I}_i^{\rm{T}} W_{c_i} \bar{I}_i{\otimes}I_{n_i})\\
&{\times}\left(\mathcal{P}_i(\bm{\xi}_i(k+1)){-}\mathcal{P}_i(\bm{\xi}_i({k}))\right)\bigg)\\
{\leq}&\sum_{i=1}^N\bigg(-\|\bm{e_{\psi_i}}(k)\|^2+\frac{1}{2}\|\bm{e_{\psi_i}}(k)\|^2\\
&+2\|\bar{C}_i^{\rm{T}}W_{c_i}\bar{I}_i\|^2\|\mathcal{P}_i(\bm{\xi}_i(k+1)){-}\mathcal{P}_i(\bm{\xi}_i({k}))\|^2\\
&+\|\bar{I}_i^{\rm{T}} W_{c_i} \bar{I}_i\|\|\mathcal{P}_i(\bm{\xi}_i(k+1)){-}\mathcal{P}_i(\bm{\xi}_i({k}))\|^2\bigg)\\
=&{-}\frac{1}{2}\|\bm{e_\psi}(k)\|^2+\sum_{i=1}^N(2\|\bar{C}_i^{\rm{T}}W_{c_i}\bar{I}_i\|^2+\|\bar{I}_i^{\rm{T}} W_{c_i} \bar{I}_i\|)\times\\
&\|\mathcal{P}_i(\bm{\xi}_i(k+1)){-}\mathcal{P}_i(\bm{\xi}_i({k}))\|^2.
\end{aligned}
\end{equation*}
Under Assumption \ref{assp.fij.lipschitz}, one has
\begin{equation*}
\begin{aligned}
&\left\|\mathcal{P}_i(\bm{\xi}_i(k+1)){-}\mathcal{P}_i(\bm{\xi}_i({k}))\right\|^2\\
=&\sum_{j=1}^{n_i}\left\|\frac{\partial f_{ij}}{\partial \bm{x}_i}(\bm{\xi_{ij}}(k+1)){-}\frac{\partial f_{ij}}{\partial \bm{x}_i}(\bm{\xi_{ij}}({k}))\right\|^2\\
{\leq}&{\sum_{j=1}^{n_i}\left(l_{ij}^2\left\|\bm{\xi}_{ij}(k+1){-}\bm{\xi}_{ij}({k})\right\|^2\right)}\\
{\leq}&\max_{j}\{l_{ij}^2\}\left\|\bm{\xi}_i(k+1){-}\bm{\xi}_i({k})\right\|^2,
\end{aligned}
\end{equation*}
substituting which back into the above inequality yields 
\begin{equation*}\label{eq.V_psi.k+1.k.1}
\begin{aligned}
&V_{\psi}(k+1){-}V_{\psi}(k)\\
{\leq}&{-}\frac{1}{2}\|\bm{e_\psi}(k)\|^2+\sum_{i=1}^N(2\|\bar{C}_i^{\rm{T}}W_{c_i}\bar{I}_i\|^2+\|\bar{I}_i^{\rm{T}} W_{c_i} \bar{I}_i\|)\times\\
&\max_{j}\{l_{ij}^2\}\|\bm{\xi}_i(k+1){-}\bm{\xi}_i({k})\|^2\\
{\leq}&{-}\frac{1}{2}\|\bm{e_\psi}(k)\|^2+\max_{i,j}\{(2\|\bar{C}_i^{\rm{T}}W_{c_i}\bar{I}_i\|^2+\|\bar{I}_i^{\rm{T}} W_{c_i} \bar{I}_i\|)l_{ij}^2\}\times\\
&\|\bm{\xi}(k+1){-}\bm{\xi}({k})\|^2.
\end{aligned}
\end{equation*}
Recalling (\ref{eq.law.xi}), it follows that
\begin{equation*}
\begin{aligned}
&V_{\psi}(k+1){-}V_{\psi}(k)\\
{\leq}&{-}\frac{1}{2}\|\bm{e_\psi}(k)\|^2+\max_{i,j}\{(2\|\bar{C}_i^{\rm{T}}W_{c_i}\bar{I}_i\|^2+\|\bar{I}_i^{\rm{T}} W_{c_i} \bar{I}_i\|)l_{ij}^2\}\times\\
&\|\Gamma(\mathcal{L}{\otimes} I_{n_\text{sum}}+A_d)\|^2\|\big(\bm{\xi}(k){-}\bm{1}_{n_\text{sum}}{\otimes}\bm{x}(k)\big)\|^2\\
{\leq}&{-}\frac{1}{2}\|\bm{e_\psi}(k)\|^2+2\max_{i,j}\{(2\|\bar{C}_i^{\rm{T}}W_{c_i}\bar{I}_i\|^2+\|\bar{I}_i^{\rm{T}} W_{c_i} \bar{I}_i\|)l_{ij}^2\}\times\\
&\|\Gamma(\mathcal{L}{\otimes} I_{n_\text{sum}}+A_d)\|^2(\|\bm{e_\xi}(k)\|^2+{n_{\text{sum}}}\|\bm{e_x}(k)\|^2)\\
=&{-}\frac{1}{2}\|\bm{e_\psi}(k)\|^2+\beta_{\psi \xi}\|\bm{e_\xi}(k)\|^2+\beta_{\psi x}\|\bm{e_x}(k)\|^2.
\end{aligned}
\end{equation*}
\end{proof}

\begin{lemma}\label{lemma.bar.x}
Suppose Assumptions \ref{assp.fij.lipschitz} and \ref{assp.Q} holds. Then, the function $V_{\bar{x}} (t)$ defined in (\ref{eq.Lya.part}) satisfies
	\begin{equation}\label{eq.lemma.bar.x}
	\begin{aligned}
	&V_{{\bar{x}}}(k+1){-}V_{{\bar{x}}}({k})\\
	{\leq}&{-}l\|\bm{e_{\bar{x}}}({k})\|^2
	+\beta_{\bar x\psi}\|\bm{e_{\psi}}(k)\|^2+
	\beta_{\bar x\xi}\|\bm{e_{\xi}}(k)\|^2\\
	&+\alpha \max\limits_{i}\{b_{1i}\}\bigg(\|{\bm{e_\psi}}(k)\|^2+\|\bm{e_{\xi}}(k)\|^2+N\|\bm{e_{\bar{x}}}(k)\|^2\bigg).
	\end{aligned}
	\end{equation}
\end{lemma}	
\begin{proof}
From the iteration of $e_{\bar{x}_i}$ in (\ref{eq.e_bar.x.i.k+1.k}), one can derive that
\begin{equation}\label{eq.V_bar.x.k+1.k.0}
\begin{aligned}
&V_{{\bar{x}}}(k+1){-}V_{{\bar{x}}}({k})\\
=&\sum_{i=1}^{N}\frac{n_i^3}{\alpha u_i^{\rm{T}}v_i}
\bigg({-}2{e_{\bar{x}_i}}({k})\frac{\alpha}{n_i^2}u_i^{\rm{T}}(I_{n_i}{\otimes}\bm{1}_{n_i}^{\rm{T}})\bm{\psi_i}(k)\\
&+\big(\frac{\alpha}{n_i^2}u_i^{\rm{T}}(I_{n_i}{\otimes}\bm{1}_{n_i}^{\rm{T}}){\bm{\psi}}_i(k)\big)^2\bigg)\\
=&\sum_{i=1}^{N}\frac{n_i}{ u_i^{\rm{T}}v_i}
\bigg({-}2{e_{\bar{x}_i}}({k})u_i^{\rm{T}}(I_{n_i}{\otimes}\bm{1}_{n_i}^{\rm{T}})\big({\bm{\psi}}_i(k){-}v_i{\otimes}\bar{\bm{\psi}}_i(k)\\
&+v_i{\otimes}\bar{\mathcal{P}}_i(\bm{\xi}_i(k)){-}v_i{\otimes}\bar{\mathcal{P}}_i(\bm{1}_{n_i}{\otimes}\bar{{X}}(k))\\
&+v_i{\otimes}\bar{\mathcal{P}}_i(\bm{1}_{n_i}{\otimes}\bar{{X}}(k)){-}v_i{\otimes}\bar{\mathcal{P}}_i(\bm{1}_{n_i}{\otimes}{\bm{x}^*})\big)\\
&+\frac{\alpha}{n_i^2}(u_i^{\rm{T}}(I_{n_i}{\otimes}\bm{1}_{n_i}^{\rm{T}}){\bm{\psi}}_i(k))^2\bigg)\\
=&\sum_{i=1}^{N}
\bigg({-}\frac{2n_i}{ u_i^{\rm{T}}v_i}{e_{\bar{x}_i}}({k})u_i^{\rm{T}}(I_{n_i}{\otimes}\bm{1}_{n_i}^{\rm{T}})\bm{e_{\psi_i}}(k)\\
&{-}2{e_{\bar{x}_i}}({k}) \bm{1}_{n_i}^{\rm{T}}(n_i\bar{\mathcal{P}}_i(\bm{\xi}_i(k)){-}n_i\bar{\mathcal{P}}_i(\bm{1}_{n_i}{\otimes}\bar{{X}}(k)))\\
&{-}2{e_{\bar{x}_i}}({k}) \bm{1}_{n_i}^{\rm{T}}(n_i\bar{\mathcal{P}}_i(\bm{1}_{n_i}{\otimes}\bar{{X}}(k)){-}n_i\bar{\mathcal{P}}_i(\bm{1}_{n_i}{\otimes}{\bm{x}^*}))\\
&+\frac{\alpha}{ u_i^{\rm{T}}v_in_i}\big(u_i^{\rm{T}}(I_{n_i}{\otimes}\bm{1}_{n_i}^{\rm{T}}){\bm{\psi}}_i(k)\big)^2\bigg),
\end{aligned}
\end{equation}
where the second equality is obtained by using the fact that
\begin{equation*}
\begin{aligned}
(I_{n_i}{\otimes}\bm{1}_{n_i}^{\rm{T}})\big(v_i{\otimes}\bar{\mathcal{P}}_i(\bm{1}_{n_i}{\otimes}{\bm{x}^*})\big)=v_i{\otimes}\big(\bm{1}_{n_i}^{\rm{T}}\bar{\mathcal{P}}_i(\bm{1}_{n_i}{\otimes}{\bm{x}^*})\big)=0.
\end{aligned}
\end{equation*}
Note that under Assumption \ref{assp.Q},
\begin{equation*}
\begin{aligned}
&{-}\sum_{i=1}^{N}2{e_{\bar{x}_i}}({k})\bm{1}_{n_i}^{\rm{T}}({n_i}\bar{\mathcal{P}}_i(\bm{1}_{n_i}{\otimes}\bar{{X}}(k)){-}{n_i}\bar{\mathcal{P}}_i(\bm{1}_{n_i}{\otimes}{\bm{x}^*}))\\
=&{-}2\bm{e_{\bar{x}}}({k})^{\rm{T}}\big(\mathcal{Q}(\bm{\bar{x}}){-}\mathcal{Q}(\bm{y}^*)\big)\\
{\leq}&{-}2l\|\bm{e_{\bar{x}}}({k})\|^2,
\end{aligned}
\end{equation*}
and that under Assumption \ref{assp.fij.lipschitz},
\begin{equation}\label{eq.ni.barPi.1}
\begin{aligned}
&\left\|n_i\bar{\mathcal{P}}_i(\bm{\xi}_i){-}
n_i\bar{\mathcal{P}}_i(\bm{1}_{n_i}{\otimes}\bar{{X}})\right\|\\
=&\bigg\|\sum\limits_{j=1}^{n_i}\bigg(\frac{\partial f_{ij}}{\partial \bm{x}_i}(\bm{\xi}_{ij}){-}\frac{\partial f_{ij}}{\partial \bm{x}_i}(\bar{{X}})\bigg)\bigg\|\\
{\leq}&\sum\limits_{j=1}^{n_i}\bigg\|\frac{\partial f_{ij}}{\partial \bm{x}_i}(\bm{\xi}_{ij}){-}\frac{\partial f_{ij}}{\partial \bm{x}_i}(\bar{{X}})\bigg\|\\
{\leq}&\sum\limits_{j=1}^{n_i}l_{ij}\|\bm{\xi}_{ij}{-}\bar{{X}}\|\\
\leq&\sqrt{\sum_{j=1}^{n_i}l_{ij}^2}\cdot\big\|\bm{\xi_i}{-}\bm{1}_{n_i}{\otimes}\bar{{X}}\big\|.
\end{aligned}
\end{equation}
Then, one can further derive from  (\ref{eq.V_bar.x.k+1.k.0}) that
\begin{equation}\label{eq.V_barx.k+1.k.1}
\begin{aligned}
&V_{{\bar{x}}}(k+1){-}V_{{\bar{x}}}({k})\\
{\leq}&2\max_{i}\{\frac{n_i\sqrt{n_i}\|u_i\|}{ u_i^{\rm{T}}v_i}\}\|\bm{e_{\bar{x}}}({k})\|\|\bm{e_{\psi}}(k)\|\\
&+\sum_{i=1}^{N}2\sqrt{n_i}\|{e_{\bar{x}_i}}({k})\| \|n_i\bar{\mathcal{P}}_i(\bm{\xi}_i(k)){-}n_i\bar{\mathcal{P}}_i(\bm{1}_{n_i}{\otimes}\bar{{X}}(k))\|\\
&{-}2l\|\bm{e_{\bar{x}}}({k})\|^2
+\sum_{i=1}^{N}\frac{\alpha\|u_i\|^2}{n_i u_i^{\rm{T}}v_i}\left\|(I_{n_i}{\otimes}\bm{1}_{n_i}^{\rm{T}}){\bm{\psi}}_i(k)\right\|^2\\
{\leq}
&\frac{l}{2}\|\bm{e_{\bar{x}}}({k})\|^2+\frac{2}{l}\max_{i}\{\frac{n_i^3\|u_i\|^2}{ (u_i^{\rm{T}}v_i)^2}\}\|\bm{e_{\psi}}(k)\|^2\\
&+\sum_{i=1}^{N}\big(\frac{l}{2}\|{e_{\bar{x}_i}}({k})\|^2+
\frac{2n_i}{l}\big\|n_i\bar{\mathcal{P}}_i(\bm{\xi}_i(k))\\
&{-}n_i\bar{\mathcal{P}}_i(\bm{1}_{n_i}{\otimes}\bar{{X}}(k))\big\|^2\big)\\
&{-}2l\|\bm{e_{\bar{x}}}({k})\|^2
+\sum_{i=1}^{N}\frac{\alpha\|u_i\|^2}{n_i u_i^{\rm{T}}v_i}\|(I_{n_i}{\otimes}\bm{1}_{n_i}^{\rm{T}}){\bm{\psi}}_i(k)\|^2\\
{\leq}&{-}l\|\bm{e_{\bar{x}}}({k})\|^2
+\frac{2}{l}\max_{i}\{\frac{n_i^3\|u_i\|^2}{ (u_i^{\rm{T}}v_i)^2}\}\|\bm{e_{\psi}}(k)\|^2\\
&+\sum_{i=1}^{N}
\frac{2n_i{\sum_{j=1}^{n_i}l_{ij}^2}}{l}\|\bm{e_{\xi_i}}(k)\|^2\\
&+\sum_{i=1}^{N}\frac{\alpha\|u_i\|^2}{n_i u_i^{\rm{T}}v_i}\|(I_{n_i}{\otimes}\bm{1}_{n_i}^{\rm{T}}){\bm{\psi}}_i(k)\|^2\\
{\leq}&{-}l\|\bm{e_{\bar{x}}}({k})\|^2
+\beta_{\bar{x}\psi}\|\bm{e_{\psi}}(k)\|^2+\beta_{\bar x\xi}\|\bm{e_{\xi}}(k)\|^2\\
&+\sum_{i=1}^{N}\frac{\alpha\|u_i\|^2}{n_i u_i^{\rm{T}}v_i}\|(I_{n_i}{\otimes}\bm{1}_{n_i}^{\rm{T}}){\bm{\psi}}_i(k)\|^2.
\end{aligned}
\end{equation}
Note that
\begin{equation}
\begin{aligned}
&(I_{n_i}{\otimes}\bm{1}_{n_i}^{\rm{T}}){\bm{\psi}}_i(k)\\
=&(I_{n_i}{\otimes}\bm{1}_{n_i}^{\rm{T}})\big({\bm{\psi}}_i(k){-}v_i{\otimes}\bar{\bm{\psi}}_i(k)+v_i{\otimes}\bar{\mathcal{P}}_i(\bm{\xi}_i(k))\\
&{-}v_i{\otimes}\bar{\mathcal{P}}_i(\bm{1}_{n_i}{\otimes}{\bm{x}^*})\big)\\
=&(I_{n_i}{\otimes}\bm{1}_{n_i}^{\rm{T}}){\bm{e_\psi}}_i(k)+\frac{ v_i\bm{1}_{n_i}^{\rm{T}}}{n_i}n_i(\bar{\mathcal{P}}_i(\bm{\xi}_i(k)){-}\bar{\mathcal{P}}_i(\bm{1}_{n_i}{\otimes}{\bm{x}^*})\big).
\end{aligned}
\end{equation}
Similar to (\ref{eq.ni.barPi.1}), under Assumption \ref{assp.fij.lipschitz},
\begin{equation}\label{eq.ni.barPi.2}
\begin{aligned}
\left\|n_i\bar{\mathcal{P}}_i(\bm{\xi}_i){-}
n_i\bar{\mathcal{P}}_i(\bm{1}_{n_i}{\otimes}\bm{x}^*)\right\|
\leq\sqrt{\sum_{j=1}^{n_i}l_{ij}^2}\cdot\big\|\bm{\xi_i}{-}\bm{1}_{n_i}{\otimes}\bm{{x}}^*\big\|.
\end{aligned}
\end{equation}

Then,
\begin{equation}
\begin{aligned}
&\|(I_{n_i}{\otimes}\bm{1}_{n_i}^{\rm{T}}){\bm{\psi}}_i(k)\|^2\\
{\leq}&\left(\sqrt{n_i}\|{\bm{e_\psi}}_i(k)\|+\frac{ \|v_i\|}{\sqrt{n_i}}\left\|n_i\bar{\mathcal{P}}_i(\bm{\xi}_i(k)){-}n_i\bar{\mathcal{P}}_i(\bm{1}_{n_i}{\otimes}{\bm{x}^*})\right\|\right)^2\\
{\leq}&\left(\sqrt{n_i}\|{\bm{e_\psi}}_i(k)\|+\frac{ \|v_i\|}{\sqrt{n_i}}\sqrt{\sum_{j=1}^{n_i}l_{ij}^2}\big\|\bm{\xi}_{i}(k){-}\bm{1}_{n_i}{\otimes}\bm{x}^*\big\|\right)^2\\
{\leq}&\bigg(\sqrt{n_i}\|{\bm{e_\psi}}_i(k)\|+\frac{ \|v_i\|}{\sqrt{n_i}}\sqrt{\sum_{j=1}^{n_i}l_{ij}^2}\|\bm{e_{\xi_{i}}}(k)\|\\
&+\frac{ \|v_i\|}{\sqrt{n_i}}\sqrt{\sum_{j=1}^{n_i}l_{ij}^2}\left\|\bm{1}_{n_i}{\otimes}(\bar{X}(k){-}\bm{x}^*)\right\|\bigg)^2\\
{\leq}&\bigg(\sqrt{n_i}\|{\bm{e_\psi}}_i(k)\|+\frac{ \|v_i\|\sqrt{\sum_{j=1}^{n_i}l_{ij}^2}}{\sqrt{n_i}}\|\bm{e_{\xi_{i}}}(k)\|\\
&+{ \|v_i\|}\sqrt{\sum_{j=1}^{n_i}l_{ij}^2}\sqrt{\max\limits_{i}\{n_i\}}\left\|\bm{e_{\bar{x}}}(k)\right\|\bigg)^2\\
{\leq}&b_{0i}\left(\|{\bm{e_\psi}}_i(k)\|^2+\|\bm{e_{\xi_{i}}}(k)\|^2+\|\bm{e_{\bar{x}}}(k)\|^2\right),
\end{aligned}
\end{equation}
where the last second inequality is derived by noting that
\begin{equation*}
\begin{aligned}
\|\bar{X}-\bm{x}^*\|
=\sqrt{\sum_{i=1}^Nn_i(\bar{x}_i{-}y_i^*)^2}
{\leq}\sqrt{\max\limits_{i}\{n_i\}}\|\bm{e_{\bar{x}}}\|.
\end{aligned}
\end{equation*}
Thus, we can derive
\begin{equation}\label{eq.b1}
\begin{aligned}
&\sum_{i=1}^{N}\frac{\alpha\|u_i\|^2}{n_i u_i^{\rm{T}}v_i}\|(I_{n_i}{\otimes}\bm{1}_{n_i}^{\rm{T}}){\bm{\psi}}_i(k)\|^2\\
{\leq}&\sum_{i=1}^{N}\alpha b_{1i}\bigg(\|{\bm{e_\psi}}_i(k)\|^2+\|\bm{e_{\xi_{i}}}(k)\|^2+\|\bm{e_{\bar{x}}}(k)\|^2\bigg)\\
{\leq}&\alpha \max\limits_{i}\{b_{1i}\}\bigg(\|{\bm{e_\psi}}(k)\|^2+\|\bm{e_{\xi}}(k)\|^2+N\|\bm{e_{\bar{x}}}(k)\|^2\bigg),
\end{aligned}
\end{equation}

Substituting (\ref{eq.b1}) into (\ref{eq.V_barx.k+1.k.1}) yields
\begin{equation}\label{eq.V_barx.k+1.k.2}
\begin{aligned}
&V_{{\bar{x}}}(k+1){-}V_{{\bar{x}}}({k})\\
{\leq}&{-}l\|\bm{e_{\bar{x}}}({k})\|^2
+\beta_{\bar x\psi}\|\bm{e_{\psi}}(k)\|^2+
\beta_{\bar x\xi}\|\bm{e_{\xi}}(k)\|^2\\
&+\alpha \max\limits_{i}\{b_{1i}\}\bigg(\|{\bm{e_\psi}}(k)\|^2+\|\bm{e_{\xi}}(k)\|^2+N\|\bm{e_{\bar{x}}}(k)\|^2\bigg).
\end{aligned}
\end{equation}
\end{proof}

\begin{lemma}\label{lemma.xi}
Suppose that Assumption \ref{assp.graph} hold. Then, the function $V_{\xi}(t)$ defined in (\ref{eq.Lya.part}) satisfies
	\begin{equation}
	\begin{aligned}
	&V_{\xi}(k+1){-}V_{\xi}(k)\\
	{\leq}&{-}\frac{1}{2}\|\bm{e_{\xi}}({k})\|^2
	+\beta_{\xi x}\big\|\bm{e_x}(k)\big\|^2\\
	&+\alpha^2\max\limits_{i}\{b_{2i}\}(\|{\bm{e_\psi}}(k)\|^2+\|\bm{e_{\xi}}(k)\|^2+N\|\bm{e_{\bar{x}}}(k)\|^2).
	\end{aligned}
	\end{equation}
\end{lemma}
	\begin{proof}
From the iteration of $\bm{e_\xi}$ in (\ref{eq.e_xi.k+1.k}), one can obtain that
		\begin{equation}\label{eq.V_psi.k+1.k.1}
		\begin{aligned}
		&V_{\xi}(k+1){-}V_{\xi}(k)\\
		=&\bm{e_\xi}({k})^{\rm{T}}(\mathcal{M}^{\rm{T}} W_\mathcal{M}\mathcal{M} {-}W_\mathcal{M} )\bm{e_{\xi}}({k})\\
		&+2\bm{e_\xi}({k})^{\rm{T}}\mathcal{M}^{\rm{T}} W_\mathcal{M}\Gamma(\mathcal{L}{\otimes I_{n_\text{sum}}}+A_d)\big(\bm{1}_{n_{\text{sum}}}{\otimes}\bm{e_x}(k)\big)\\
		&+2\bm{e_\xi}({k})^{\rm{T}}\mathcal{M}^{\rm{T}} W_\mathcal{M}\big(\bm{1}_{n_\text{sum}}{\otimes}(\alpha h(k))\big)\\
		&+\big(\bm{1}_{n_{\text{sum}}}{\otimes}\bm{e_x}(k)\big)^{\rm{T}}\big(\Gamma(\mathcal{L}{\otimes I_{n_\text{sum}}}+A_d)\big)^{\rm{T}}W_\mathcal{M}{\times}\\
		&\Gamma(\mathcal{L}{\otimes I_{n_\text{sum}}}+A_d)\big(\bm{1}_{n_{\text{sum}}}{\otimes}\bm{e_x}(k)\big)\\
		&+2\big(\bm{1}_{n_{\text{sum}}}{\otimes}\bm{e_x}(k)\big)^{\rm{T}}\big(\Gamma(\mathcal{L}{\otimes I_{n_\text{sum}}}+A_d)\big)^{\rm{T}}W_\mathcal{M}\times\\
		&\big(\bm{1}_{n_\text{sum}}{\otimes}(\alpha h(k))\big)\\
		&+\big(\bm{1}_{n_\text{sum}}{\otimes}(\alpha h(k))\big)^{\rm{T}}W_{\mathcal{M}}\big(\bm{1}_{n_\text{sum}}{\otimes}(\alpha h(k))\big)\\
		{\leq}&-\bm{e_\xi}({k})^{\rm{T}}\bm{e_{\xi}}({k})\\
		&+2\sqrt{n_{\text{sum}}}\|\mathcal{M}^{\rm{T}} W_\mathcal{M}\Gamma(\mathcal{L}{\otimes I_{n_\text{sum}}}+A_d)\|\|\bm{e_\xi}({k})\|\|\bm{e_x}(k)\|\\
		&+2\alpha\sqrt{n_{\text{sum}}}\|\mathcal{M}^{\rm{T}} W_\mathcal{M}\|\|\bm{e_\xi}({k})\|\|h(k)\|\\
		&+n_{\text{sum}}\|W_\mathcal{M}\|\|\Gamma(\mathcal{L}{\otimes I_{n_\text{sum}}}+A_d)\|^2\big\|\bm{e_x}(k)\big\|^2\\
		&+2\alpha{n_{\text{sum}}}\|\Gamma(\mathcal{L}{\otimes I_{n_\text{sum}}}+A_d)\|\|W_\mathcal{M}\|\big\|\bm{e_x}(k)\big\|\|h(k)\|\\
		&+\alpha^2{n_{\text{sum}}}\|W_{\mathcal{M}}\|\|h(k)\|^2\\
		{\leq}&-\bm{e_\xi}({k})^{\rm{T}}\bm{e_{\xi}}({k})+\frac{1}{4}\|\bm{e_{\xi}}({k})\|^2\\
		&+  4{n_{\text{sum}}}\|\mathcal{M}^{\rm{T}} W_\mathcal{M}\|^2\|\Gamma(\mathcal{L}{\otimes I_{n_\text{sum}}}+A_d)\|^2\|\bm{e_x}(k)\|^2\\
		&+\frac{1}{4}\|\bm{e_{\xi}}({k})\|^2+4\alpha^2{n_{\text{sum}}}\|\mathcal{M}^{\rm{T}} W_\mathcal{M}\|^2\|h(k)\|^2\\
		&+2n_{\text{sum}}\|W_\mathcal{M}\|\|\Gamma(\mathcal{L}{\otimes I_{n_\text{sum}}}+A_d)\|^2\big\|\bm{e_x}(k)\big\|^2\\
		&+2\alpha^2{n_{\text{sum}}}\|W_{\mathcal{M}}\|\|h(k)\|^2\\
		=&{-}\frac{1}{2}\|\bm{e_{\xi}}({k})\|^2
		+\beta_{\xi x}\big\|\bm{e_x}(k)\big\|^2\\
		&+\alpha^2{n_{\text{sum}}}\big(4\|\mathcal{M}^{\rm{T}} W_\mathcal{M}\|^2+2\|W_\mathcal{M}\|\big)\|h(k)\|^2.
		\end{aligned}
		\end{equation}

		Note that
		\begin{equation}
		\begin{aligned}
		&{n_{\text{sum}}}\big(4\|\mathcal{M}^{\rm{T}} W_\mathcal{M}\|^2+2\|W_\mathcal{M}\|\big)\|h(k)\|^2\\
		{\leq}&{n_{\text{sum}}}\big(4\|\mathcal{M}^{\rm{T}} W_\mathcal{M}\|^2+2\|W_\mathcal{M}\|\big)\times\\
		&\sum_{i=1}^N\|\frac{\bm{1}_{n_i}u_i^{\rm{T}}}{n_i^2}\|^2\|(I_{n_i}{\otimes}\bm{1}_{n_i}^{\rm{T}}){\bm{\psi}}_i(k)\|^2\\
		{\leq}& \max\limits_{i}\{b_{2i}\}(\|{\bm{e_\psi}}(k)\|^2+\|\bm{e_{\xi}}(k)\|^2+N\|\bm{e_{\bar{x}}}(k)\|^2).
		\end{aligned}
		\end{equation}
		
		It follows that
		\begin{equation}
		\begin{aligned}
		&V_{\xi}(k+1){-}V_{\xi}(k)\\
		{\leq}&{-}\frac{1}{2}\|\bm{e_{\xi}}({k})\|^2
		+\beta_{\xi x}\big\|\bm{e_x}(k)\big\|^2\\
		&+\alpha^2\max\limits_{i}\{b_{2i}\}(\|{\bm{e_\psi}}(k)\|^2+\|\bm{e_{\xi}}(k)\|^2+N\|\bm{e_{\bar{x}}}(k)\|^2).
		\end{aligned}
		\end{equation}
	\end{proof}

\begin{lemma}\label{lemma.x}
Suppose that Assumption \ref{assp.graph} holds. Then, the function $V_x(t)$ defined in (\ref{eq.Lya.part}) satisfies
	\begin{equation}\label{eq.lemma.V_x}
	\begin{aligned}
	&V_{x}(k+1)-V_x(k)\\
	{\leq}&-\frac{1}{2}\bm{e_{x}}^{\rm{T}}({k})\bm{e_{x}}({k})\\
	&+\alpha^2\max\limits_{i}\{b_{3i}\}(\|{\bm{e_\psi}}(k)\|^2+\|\bm{e_{\xi}}(k)\|^2+N\|\bm{e_{\bar{x}}}(k)\|^2).
	\end{aligned}
	\end{equation}
\end{lemma}
\begin{proof}
We have
\begin{equation*}\label{eq.V_x.k+1.k}
\begin{aligned}
&V_{x}(k+1)-V_x(k)\\
=&\sum_{i=1}^{N}\bigg(\bm{e_{x_i}}^{\rm{T}}({k})\big(\bar{R}_i^{\rm{T}} W_{R_i}\bar{R}_i {-}W_{R_i}\big)\bm{e_{x_i}}({k})\\
&{-}\frac{2{\alpha}}{n_i}\bm{e_{x_i}}^{\rm{T}}(k)\bar{R}_{i}^{\rm{T}}W_{R_i}\bar{I}_{u_i}{(I_{n_i}{\otimes}\bm{1}_{n_i}^{\rm{T}})}{\bm{\psi}}_i(k)\\
&+\frac{{\alpha}^2}{n_i^2}\big(\bar{I}_{u_i}{(I_{n_i}{\otimes}\bm{1}_{n_i}^{\rm{T}})}{\bm{\psi}}_i(k)\big)^{\rm{T}}W_{R_i}\bar{I}_{u_i}{(I_{n_i}{\otimes}\bm{1}_{n_i}^{\rm{T}})}{\bm{\psi}}_i(k)\bigg)\\
\leq&\sum_{i=1}^{N}\bigg(-\bm{e_{x_i}}^{\rm{T}}({k})\bm{e_{x_i}}({k})+\frac{1}{2}\bm{e_{x_i}}^{\rm{T}}({k})\bm{e_{x_i}}({k})\\
&+\frac{2{\alpha^2}}{n_i^2}\|\bar{R}_{i}^{\rm{T}}W_{R_i}\bar{I}_{u_i}\|^2\|{(I_{n_i}{\otimes}\bm{1}_{n_i}^{\rm{T}})}{\bm{\psi}}_i(k)\|^2\\
&+\frac{{\alpha}^2}{n_i^2}\|\bar{I}_{u_i}^{\rm{T}}W_{R_i}\bar{I}_{u_i}\|\|{(I_{n_i}{\otimes}\bm{1}_{n_i}^{\rm{T}})}{\bm{\psi}}_i(k)\|^2\bigg)\\
{\leq}&-\frac{1}{2}\bm{e_{x}}^{\rm{T}}({k})\bm{e_{x}}({k})\\
&+\sum_{i=1}^N\alpha^2b_{3i}(\|{\bm{e_\psi}}_i(k)\|^2+\|\bm{e_{\xi_{i}}}(k)\|^2+\|\bm{e_{\bar{x}}}(k)\|^2)\\
{\leq}&-\frac{1}{2}\bm{e_{x}}^{\rm{T}}({k})\bm{e_{x}}({k})\\
&+\alpha^2\max\limits_{i}\{b_{3i}\}(\|{\bm{e_\psi}}(k)\|^2+\|\bm{e_{\xi}}(k)\|^2+N\|\bm{e_{\bar{x}}}(k)\|^2).
\end{aligned}
\end{equation*}
\end{proof}

\vspace{1.5ex}


\end{document}